\documentclass[reqno]{amsart}
\usepackage{amsfonts}
\usepackage{amsmath,amsthm,mathrsfs}
\usepackage{amssymb}
\usepackage{enumitem}
\usepackage{tikz}
\usepackage{bm}

\usepackage[normalem]{ulem}

\usepackage[colorlinks=true,linkcolor=blue,citecolor=blue,urlcolor=blue,pdfborder={0 0 0}]{hyperref}

\usepackage{graphicx}
\usepackage{subcaption}

\usepackage{xcolor}

\theoremstyle{plain}
\newtheorem{thm}{Theorem}[section]
\newtheorem{prop}[thm]{Proposition}
\newtheorem{lemma}[thm]{Lemma}
\newtheorem{cor}[thm]{Corollary}
\newtheorem{claim}[thm]{Claim}

\theoremstyle{definition}
\newtheorem*{defn}{Definition}

\newcommand{\e}{\varepsilon}
\newcommand{\C}{\mathbb{C}}
\newcommand{\EE}{\mathcal{E}}
\newcommand{\R}{\mathbb{R}}
\newcommand{\uu}{\mathbf{u}}
\newcommand{\vv}{\mathbf{v}}
\newcommand{\ww}{\mathbf{w}}

\newcommand{\1}{\mathbf{1}}
\newcommand{\fout}{f^{\mathrm{out}}}
\DeclareMathOperator{\area}{\mathrm{area}}
\DeclareMathOperator{\diam}{\mathrm{diam}}
\DeclareMathOperator{\Exp}{\mathbb{E}}
\DeclareMathOperator{\Span}{\mathrm{Span}}
\DeclareMathOperator{\gap}{\mathrm{gap}}
\DeclareMathOperator{\strength}{\mathrm{strength}}
\DeclareMathOperator{\Int}{\mathrm{Int}}
\DeclareMathOperator{\Prob}{\mathbb{P}}
\DeclareMathOperator{\tlog}{\mathrm{tlog}}
\DeclareMathOperator{\conv}{\mathrm{conv}}
\DeclareMathOperator{\dist}{\mathrm{dist}}

\newcommand{\mcl}{\mathcal}

\newcommand{\DD}{\mathbb{D}}

\title{The Dirichlet problem for orthodiagonal maps}
\author{Ori Gurel-Gurevich, Daniel C. Jerison and Asaf Nachmias}

\AtEndDocument{\bigskip{\footnotesize%
\par
\textsc{Ori Gurel-Gurevich} \par
\textsc{Hebrew University of Jerusalem} \par
\textit{Email:} \texttt{Ori.Gurel-Gurevich@mail.huji.ac.il} \par
\addvspace{\medskipamount}

\textsc{Daniel C. Jerison} \par
\textsc{Department of Mathematical Sciences, Tel Aviv University, Tel Aviv 69978, Israel} \par
  \textit{Email:} \texttt{jerison@mail.tau.ac.il}, \texttt{dcjerison@gmail.com} \par
  \addvspace{\medskipamount}

\textsc{Asaf Nachmias} \par
\textsc{Department of Mathematical Sciences, Tel Aviv University, Tel Aviv 69978, Israel} \par
  \textit{Email:} \texttt{asafnach@tauex.tau.ac.il} \par
  \addvspace{\medskipamount}

}}

\begin{document}
\begin{abstract} We prove that the discrete harmonic function corresponding to smooth Dirichlet boundary conditions on orthodiagonal maps, that is, plane graphs having quadrilateral faces with orthogonal diagonals, converges to its continuous counterpart as the mesh size goes to $0$. This provides a convergence statement for discrete holomorphic functions, similar to the one obtained by Chelkak and Smirnov \cite{ChelkakSmirnov2011} for isoradial graphs. We observe that by the double circle packing theorem \cite{BriSch93}, any finite, simple, 3-connected planar map admits an orthodiagonal representation.

Our result improves the work of Skopenkov \cite{S13} and Werness \cite{W15} by dropping \emph{all} regularity assumptions required in their work and providing effective bounds. In particular, no bound on the vertex degrees is required. Thus, the result can be applied to models of random planar maps that with high probability admit orthodiagonal representation with mesh size tending to $0$. In a companion paper \cite{GGJN19}, we show that this can be done for the discrete \emph{mating-of-trees} random map model of Duplantier, Gwynne, Miller and Sheffield \cite{DMS14,GMS17}.
\end{abstract}

\maketitle

\section{Introduction}

Discrete complex analysis is a powerful tool in the study of two-dimensional statistical physics. It has been employed to prove the conformal invariance of the scaling limit of critical percolation \cite{SmirnovPerc} and the critical Ising/FK model \cite{SmirnovIsing,ChelkakSmirnovIsing}, see Smirnov's ICM survey \cite{SmirnovICM}.
The high-level program of such proofs is to 1) find a model-dependent function (the so-called \emph{discrete parafermionic observable}) on the lattice which satisfies some discrete version of the Cauchy-Riemann equations; 2) use discrete complex analysis to show that as the lattice's mesh size tends to $0$, the discrete observable converges to a continuous holomorphic function; 3) identify this function uniquely by its boundary values. The results obtained this way include some of the most remarkable breakthroughs in contemporary probability theory.

In this paper we address the second part of the program above, namely, the convergence of discrete harmonic or holomorphic functions to their continuous counterparts. This study has been performed on the square lattice \cite{Courant28} as well as on \emph{rhombic lattices} \cite{ChelkakSmirnov2011}, which are plane graphs such that each inner face is a rhombus. Because not every quadrangulation can be embedded in $\C$ as a rhombic lattice \cite{KS05}, Smirnov asked, ``can we always find another embedding with a sufficiently nice version of discrete complex analysis?'' \cite[Section 6, Question 1]{SmirnovICM}.

A broader class than the rhombic lattices are the \emph{orthodiagonal maps}: plane graphs whose inner faces are quadrilaterals with orthogonal diagonals (see Section \ref{sec:terminology}). One can ask which planar maps are \emph{representable} by a rhombic lattice or by an orthodiagonal map; see Section \ref{sec:orthorep}. For rhombic lattices the answer is the \emph{isoradial graphs}. Unfortunately, these do not include all finite simple triangulations \cite{KS05}. By contrast, every finite simple triangulation has an orthodiagonal representation which can be constructed using circle packing \cite[Remark 5]{M01}. We elaborate on this in Section \ref{sec:orthorep} and observe that the double circle packing theorem \cite{BriSch93,Thurston} provides an orthodiagonal representation for any simple, 3-connected finite planar map.

Skopenkov \cite{S13} proved a convergence result for the Dirichlet problem on orthodiagonal maps under certain local and global regularity conditions. Werness \cite{W15} improved this result to assume only local regularity. See also the work of Dubejko \cite{D99}. All these works require a uniform bound on the vertex degrees of the maps.

The main result of this paper, Theorem \ref{main_thm}, is a convergence statement that has \emph{no} regularity assumptions of any kind. In particular, our result applies even when the vertex degrees are not uniformly bounded. Our only condition is that the maximal edge length of the map tends to $0$. As well, our proof avoids compactness arguments and thus provides an effective bound for the convergence.

Removing the regularity assumptions is not just mathematically pleasing; rather, it provides a framework for the study of discrete complex analysis on \emph{random planar maps} \cite{LeGICM, Miermont14,CurienPeccot}. In order to apply Theorem \ref{main_thm} to a given random map model, one has to verify the maximal edge length condition above.
This condition is believed to hold in all natural random map models, though proving it for a random simple triangulation on $n$ vertices is considered an important open problem (see \cite[Section 6]{LeGICM}). In the companion paper \cite{GGJN19} we show that it indeed holds for the discrete \emph{mating-of-trees} random map model of Duplantier-Gwynne-Miller-Sheffield \cite{DMS14,GMS17}. Hence Theorem \ref{main_thm} can be applied to this model; see \cite{GGJN19}.

There has been a great deal of interest in recent years in studying statistical physics models, such as percolation, Ising/FK and the self-avoiding walk, on random planar maps \cite{AngelPerc,CurienCaraceni,GwynneFK,GwynneMillerSelfAvoiding,GwynneMillerPercolation}. The behavior of these models at their critical temperature is mysteriously related via the KPZ correspondence to their behavior on the usual square or triangular lattices \cite{DS11,Garban13}. A very ambitious program is to rigorously relate the behavior of a statistical physics model in the random planar map setting (where in many cases the model is tractable) to its behavior on a regular lattice. We hope that the framework for discrete complex analysis on random planar maps  that we provide in this paper will be useful for this endeavor.

Below is the statement of our main theorem. Even though some of the notation has not been defined, the conclusion should be clear: the discrete harmonic function is close to the continuous one when the mesh size is small. We have gathered the necessary definitions required to parse this theorem in Section \ref{sec:terminology} below. For the experienced reader, we remark that our discrete harmonic functions are with respect to the canonical edge weights associated with the map rather than unit weights.

\begin{thm} \label{main_thm}
Let $\Omega \subset \R^2$ be a bounded simply connected domain, and let $g: \R^2 \to \R$ be a $C^2$ function. Given $\e, \delta\in (0,\diam(\Omega))$, let $G = (V^\bullet \sqcup V^\circ, E)$ be a finite orthodiagonal map with maximal edge length at most $\e$ such that the Hausdorff distance between $\partial G$ and $\partial \Omega$ is at most $\delta$. Let $h_c: \overline{\Omega} \to \R$ be the solution to the continuous Dirichlet problem on $\Omega$ with boundary data $g$, and let $h_d: V^\bullet \to \R$ be the solution to the discrete Dirichlet problem on $\Int(V^\bullet)$ with boundary data $g|_{\partial V^\bullet}$. Set
\[
C_1 = \sup_{x \in \widetilde{\Omega}} |\nabla g(x)|\, , \qquad C_2 = \sup_{x \in \widetilde{\Omega}} \|Hg(x)\|_2
\]
where $\widetilde{\Omega} = \conv(\overline{\Omega} \cup \widehat{G})$. Then there is a universal constant $C < \infty$ such that for all $x \in V^\bullet \cap \overline{\Omega}$,
\[
|h_d(x) - h_c(x)| \leq \frac{C \diam(\Omega) (C_1 + C_2 \e)}{\log^{1/2}(\diam(\Omega) / (\delta \vee \e))} \, .
\]
\end{thm}

\noindent {\em Remark.} A consequence of Theorem \ref{main_thm} is that if the sequence $G_n$ of orthodiagonal maps approximates $\Omega$, meaning that the maximal edge length in $G_n$ tends to zero and the boundaries $\partial G_n$ converge to $\partial \Omega$ in a suitable sense, then the exit measure of the weighted random walk defined in Section \ref{sec:terminology} on the primal network of $G_n$ converges to the harmonic measure on $\Omega$. For a precise statement we refer the reader to \cite[Corollary 5.7]{S13}, whose proof can easily be modified to use Theorem \ref{main_thm} in place of \cite[Theorem 1.2]{S13} and thereby remove the local and global regularity conditions on the maps $G_n$. \\

\subsection{Notations and terminology}\label{sec:terminology} The closure, boundary, and diameter of $\Omega$ are denoted by $\overline{\Omega}$, $\partial \Omega$, and $\diam(\Omega)$. The convex hull of a set $S$ is $\conv(S)$. By $\nabla g$ and $Hg$ we mean the gradient and the Hessian matrix of $g$. The notation $|\nabla g|$ means the Euclidean norm, and the notation $\|Hg\|_2$ means the $L^2 \to L^2$ operator norm (which is also the spectral radius since $Hg$ is symmetric). The Hausdorff distance between two sets $S,T$ is the infimum of all $r > 0$ such that each $s \in S$ is within distance $r$ from some element of $T$ and each $t \in T$ is within distance $r$ from some element of $S$. The solution to the continuous Dirichlet problem on $\Omega$ with boundary data $g$ is the unique continuous function $h_c: \overline{\Omega} \to \R$ such that $h_c = g$ on $\partial \Omega$ and $h_c$ is harmonic on $\Omega$.

A plane graph is a graph $G = (V,E)$ with a fixed proper embedding in the plane. We frequently identify the vertices and edges of the graph with the points and curves in the plane of the embedding.
The faces of a finite plane graph $G$ are the connected components of $\R^2$ minus the edges and vertices of $G$. All but one of the faces are bounded; these are called inner faces, and the unbounded face is called the outer face. The degree of a face is the number of edges in its boundary (with an edge counted twice if the face borders it from both sides).
\begin{defn}
A finite \textbf{orthodiagonal map} is a finite connected plane graph in which:
\begin{itemize}
\item Each edge is a straight line segment;
\item Each inner face is a quadrilateral with orthogonal diagonals; and
\item The boundary of the outer face is a simple closed curve.	
\end{itemize}
\noindent We allow non-convex quadrilaterals, whose diagonals do not intersect. See Figure \ref{orthomap}. Orthodiagonal maps are called ``orthogonal lattices'' by \cite{S13,W15} and ``semi-critical maps'' by \cite{M01}. The requirement that the boundary of the outer face must be simple is not severe: in case $G$ satisfies the other conditions, we may consider the blocks of $G$ (i.e.\ maximal $2$-connected components) separately. For each block $H$, the boundary of its outer face is a simple closed curve \cite[Proposition 4.2.5]{D17} and its inner faces are all inner faces of $G$ (by Lemma \ref{block faces}, proved below), so $H$ is an orthodiagonal map.
\end{defn}

\begin{figure}
\centering
\includegraphics[width=.4\textwidth]{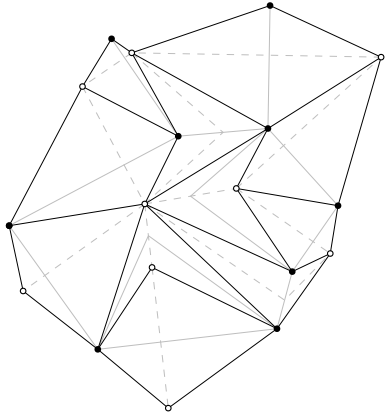}
\caption{An orthodiagonal map. Primal vertices are represented by solid disks $\bullet$ and dual vertices by hollow disks $\circ$. The edges of the orthodiagonal map itself are drawn in black. Edges of the primal graph are drawn with solid gray lines, and edges of the dual graph are drawn with dashed gray lines.}
\label{orthomap}	
\end{figure}

Every finite orthodiagonal map $G$ is a bipartite graph: if it contained an odd cycle, then there would be a face of odd degree inside the cycle. Thus we may write $G = (V^\bullet \sqcup V^\circ, E)$ where $V^\bullet \sqcup V^\circ$ is the bipartition of the vertices. We will use the notation $Q = [v_1,w_1,v_2,w_2]$ to denote an inner face of $G$. This means that the boundary of $Q$ passes in order through the vertices $v_1,w_1,v_2,w_2$ when traversed counterclockwise, and that $v_1,v_2 \in V^\bullet$ while $w_1,w_2 \in V^\circ$. For each inner face $Q = [v_1,w_1,v_2,w_2]$ of $G$, we draw a \textbf{primal edge} $e^\bullet_Q$ between $v_1$ and $v_2$ and a \textbf{dual edge} $e^\circ_Q$ between $w_1$ and $w_2$, as follows. If the line segment $v_1 v_2$ is contained inside $Q$ (except for the endpoints), then $e^\bullet_Q$ is this segment. If not, we let $p$ be the midpoint of the segment $w_1 w_2$ and draw $e^\bullet_Q$ as the union of the segments $v_1 p$ and $p v_2$. Similarly, $e^\circ_Q$ is either $w_1 w_2$, if that segment is contained in $Q$, or $w_1 q \cup q w_2$ otherwise, where $q$ is the midpoint of $v_1 v_2$. By this construction, $e^\bullet_Q$ and $e^\circ_Q$ are contained in $Q$ and intersect exactly once. The primal and dual edges are drawn in gray in Figure \ref{orthomap}.

The \textbf{primal graph} associated with $G$ is $G^\bullet = (V^\bullet, E^\bullet)$, where $E^\bullet = \bigcup_Q e_Q^\bullet$. The \textbf{dual graph} is $G^\circ = (V^\circ, E^\circ)$, where $E^\circ = \bigcup_Q e_Q^\circ$. These are plane graphs, not necessarily simple, and $e_Q^\bullet \leftrightarrow e_Q^\circ$ is a bijection between $E^\bullet$ and $E^\circ$. Despite the names, $G^\bullet$ and $G^\circ$ are not quite plane duals, as we will see. It is not difficult to see (Lemma \ref{lem:connected}) that $G^\bullet$ and $G^\circ$ are connected.

The \textbf{boundary} of $G$, denoted $\partial G$, is the (topological) boundary of its outer face. The \textbf{interior} of $G$, denoted $\Int(G)$, is the open subset of $\R^2$ enclosed by $\partial G$, so that $\widehat{G} := \Int(G) \cup \partial G$ is the union of the closures of the inner faces. The \textbf{boundary vertices} of $G^\bullet$ and $G^\circ$ are $\partial V^\bullet = \partial G \cap V^\bullet$ and $\partial V^\circ = \partial G \cap V^\circ$. The \textbf{interior vertices} of $G^\bullet$ and $G^\circ$ are $\Int(V^\bullet) = \Int(G) \cap V^\bullet = V^\bullet \setminus \partial V^\bullet$ and $\Int(V^\circ) = \Int(G) \cap V^\circ = V^\circ \setminus \partial V^\circ$. The graphs $G^\bullet$ and $G^\circ$ are nearly plane duals of each other, except that the outer face of $G^\circ$ contains all the boundary vertices of $G^\bullet$, and the outer face of $G^\bullet$ contains all the boundary vertices of $G^\circ$. When exact duality is required in Section \ref{Equicontinuity section}, we will consider augmented versions of $G^\bullet$ and $G^\circ$.

Orthodiagonal maps come with a ``conformally natural'' set of positive edge weights, which we now define. These weights were defined by Duffin \cite{D68} and independently by Dubejko \cite{D97} and are intimately related to discrete holomorphic functions; we discuss this in Section \ref{sec:previousresults}.

\begin{defn}
Let $Q = [v_1,w_1,v_2,w_2]$ be an inner face of an orthodiagonal map $G = (V^\bullet \sqcup V^\circ, E)$. The conductances of the primal edge $e_Q^\bullet$ and the dual edge $e_Q^\circ$ are given by
\begin{equation} \label{conductances}
c(e^\bullet_Q) = \frac{|w_1 w_2|}{|v_1 v_2|} \, , \qquad c(e^\circ_Q) = \frac{|v_1 v_2|}{|w_1 w_2|} \, .
\end{equation}
We emphasize that the weights $c(e^\bullet_Q)$, $c(e^\circ_Q)$ are determined by the Euclidean distance
between the endpoints of the edges even when $e^\bullet_Q$ or $e^\circ_Q$ is ``bent'' (when $Q$ is concave).
\end{defn}

The \textbf{primal network} and \textbf{dual network} associated with $G$ are $(G^\bullet, c)$ and $(G^\circ, c)$, with edge conductances $c$ as above. A function $h: V^\bullet \to \R$ is called \textbf{discrete harmonic} at $v_1 \in V^\bullet$ if
\begin{equation}\label{def:discreteharmonicortho} h(v_1) \sum_{Q=[v_1,w_1,v_2,w_2]} c(e^\bullet_Q) =  \sum_{Q=[v_1,w_1,v_2,w_2]} c(e^\bullet_Q) h(v_2) \, ,\end{equation}
where the sum is over all faces $Q$ which are incident to $v_1$. (The notion of a discrete harmonic function on a network will be discussed in full generality in Section \ref{Electric networks}.) Duffin \cite{D68} and indepedently Dubejko \cite{D97} observed that the function assigning to each vertex its horizontal or vertical coordinate is discrete harmonic on $\Int(V^\bullet)$; in other words, the random walk on the primal network of an orthodiagonal map is a martingale on its interior vertices. See Proposition \ref{martingale}. This fundamental property is the reason the weights \eqref{conductances} are canonical.

The solution to the discrete Dirichlet problem on $\Int(V^\bullet)$ with boundary data $g|_{\partial V^\bullet}$ is the unique function $h_d: V^\bullet \to \R$ such that $h_d = g$ on $\partial V^\bullet$ and $h_d$ is discrete harmonic on $\Int(V^\bullet)$. It is well-known and not difficult that such a solution always exists and is unique (Proposition \ref{discrete harmonic}).

\subsection{Previous work}\label{sec:previousresults}
The study of discrete harmonic functions on the square lattice is classical. We mention only the important paper of Courant, Friedrichs and Lewy \cite{Courant28} which proved that the solutions to the Dirichlet problem for finer and finer discretizations of an elliptic operator converge to the appropriate continuous solution as the mesh size tends to zero.

Our results have a natural interpretation in terms of discrete analytic functions, which were first studied on the square lattice by Isaacs \cite{I41} and Lelong-Ferrand \cite{F44,LF55}. Duffin \cite{D68} showed that much of the theory generalizes nicely to the setting of rhombic lattices, which we recall are plane graphs whose inner faces are all rhombi. Since rhombic lattices are a subclass of orthodiagonal maps, we may describe them using the terminology from Section \ref{sec:terminology}.

Let $Q = [v_1,w_1,v_2,w_2]$ be an inner face of an orthodiagonal map $G = (V^\bullet \sqcup V^\circ, E)$. Duffin calls a function $f: V^\bullet \sqcup V^\circ \to \C$ {\textbf{discrete analytic}} at $Q$ if
\[
\frac{f(v_2) - f(v_1)}{v_2 - v_1} = \frac{f(w_2) - f(w_1)}{w_2 - w_1} \, .
\]
We say that $f$ is discrete analytic on $G$ if it is discrete analytic at every inner face of $G$. This definition of discrete analyticity turns out to be very natural and applicable, see \cite{SmirnovICM}.

The connection to this paper is Duffin's observation \cite{D68} that the real part of any discrete analytic function on $G$, restricted to $V^\bullet$, is discrete harmonic on $\Int(V^\bullet)$ with respect to the canonical weights \eqref{conductances}, i.e., it satisfies \eqref{def:discreteharmonicortho}. (Duffin proved this statement for rhombic lattices but recognized that the proof carries over to the more general setting of orthodiagonal maps.) Thus, convergence statements such as Theorem \ref{main_thm} can be used to prove convergence statements for discrete analytic functions to their continuous counterparts. \\

Around the turn of the millennium, it was recognized by Mercat \cite{M98,M01}, Kenyon \cite{K00,K02} and Smirnov \cite{SmirnovPerc} that discrete analytic functions could be used to prove important properties of probabilistic models such as conformal invariance. This discovery led to a rejuvenation of the field of discrete complex analysis. Like Duffin, Mercat \cite{M01} and Kenyon \cite{K02} selected rhombic lattices as the ideal ground on which to develop a discrete analogue to the familiar continuous theory. Mercat referred to rhombic lattices as ``critical maps,'' while Kenyon used the term ``isoradial embedding'' to describe the graphs $G^\bullet$ and $G^\circ$ associated with a rhombic lattice $G$. Both ``rhombic lattice'' and ``isoradial graph'' are now standard terms in the literature. Mercat also considered ``semi-critical maps,'' which are essentially the same as our orthodiagonal maps, and observed that they can be generated from circle packings of triangulations by the procedure we described in Proposition \ref{prop:orthodiagonal-triangulation}.

On rhombic lattices, convergence of solutions to the discrete Dirichlet problem was proved by Chelkak and Smirnov \cite{ChelkakSmirnov2011} under a local regularity assumption, namely, that each rhombus's angles are bounded away from $0$ and $\pi$ uniformly. (Of course, the requirement that all edges of a rhombic lattice have the same length is a very strong form of global regularity.) Skopenkov \cite{S13} proved a similar convergence result for orthodiagonal maps (which he calls ``orthogonal lattices''). Werness \cite{W15} succeeded in removing the global regularity assumption from Skopenkov's result at the mild cost of slightly changing the local regularity assumption. It is in this context that we prove Theorem \ref{main_thm}, which is strictly stronger than the convergence results of Skopenkov and Werness in that it removes all regularity assumptions and provides a quantitative bound. It also follows from our result that one can drop the local regularity condition of Chelkak-Smirnov \cite{ChelkakSmirnov2011} in their statements on convergence of discrete harmonic functions, that is, Proposition 3.3 and Theorem 3.10 in \cite{ChelkakSmirnov2011}. (Note however that we prove only $C^0$ convergence rather than $C^1$ convergence and that we require smooth boundary data.) \\

Lastly, we wish to point out the papers \cite{D97,D99} by Dubejko. While unaware of Duffin's results \cite{D68}, Dubejko \cite{D97} rediscovered the weights \eqref{conductances} and observed that the function assigning to each vertex its horizontal or vertical coordinate is discrete harmonic on $\Int(V^\bullet)$. In \cite{D99} Dubejko proved a result similar to Skopenkov's \cite{S13}, that is, convergence on orthodiagonal maps of the solution of the discrete Dirichlet problem to its continuous counterpart assuming global and local regularity assumptions, although his regularity assumptions are strictly stronger than Skopenkov's. Other differences are that Dubejko requires that the graph $G^\bullet$ be a triangulation and that the quadrilateral faces of $G$ be convex. On the other hand, Dubejko treats the Dirichlet problem for Poisson's equation and not merely Laplace's equation.

\subsection{Organization of paper and proof overview}\label{sec:proofoverview} In Section \ref{sec:orthorep} below we show how circle packing and double circle packing can be used to generate orthodiagonal representations of planar maps. This is well-known \cite[Remark 5]{M01} in the case of finite planar triangulations, and likely also its generalization using double circle packing for finite 3-connected planar maps, though we were unable to find the latter in the literature. We provide the full details in Section \ref{sec:orthorep} for completeness, noting that this section is independent of the rest of the paper.

In Sections \ref{preliminaries} through \ref{main proof} we prove Theorem \ref{main_thm}. The proof uses the theory of electric networks. We consider the difference $h_d - h_c$ as a function on the vertex set $V^\bullet$. There are two main steps: first we obtain an $L^2$ bound (i.e., an energy bound) on $h_d - h_c$, and then we upgrade it to the $L^\infty$ bound in the statement of the theorem. We begin in Section \ref{preliminaries} with some preliminary lemmas, a few of which fill in details to justify assertions that we made in Section \ref{sec:terminology}.

To prove the energy bound, Proposition \ref{energy convergence}, we use $h_c$ to construct two functions on the edges of the graph $G^\bullet$. The first function is the discrete gradient in $G^\bullet$ of the restriction of $h_c$ to $V^\bullet$. The second function is defined by restricting the harmonic conjugate of $h_c$ to the dual vertices $V^\circ$, taking the discrete gradient in $G^\circ$, and finally recovering a function on the edges of $G^\bullet$ using duality. A direct computation shows that these two functions are close to each other. It turns out that this immediately implies the energy bound on $h_d - h_c$, for a reason coming from the abstract theory of electric networks with multiple sources and sinks. This theory is a slight generalization of the linear-algebraic formulation of electric networks developed in \cite{BLPS01}. We first build up the abstract theory in Section \ref{Electric networks} and then present the argument above in Section \ref{Energy convergence}.

To finish the proof of Theorem \ref{main_thm}, the main ingredients are two electric resistance estimates which are proved in Section \ref{sec:resistance}. One of these estimates is used in Section \ref{Equicontinuity section} to prove a smoothness result for $h_d$, Proposition \ref{equicontinuity}. With this statement in hand, we proceed to the proof of Theorem \ref{main_thm} in Section \ref{main proof}. Here is an outline of the argument. Suppose for the sake of contradiction that there is a vertex $x \in \Int(V^\bullet)$ away from the boundary for which $|h_d(x) - h_c(x)|$ is large. By the smoothness of $h_d$ proved in Proposition \ref{equicontinuity} (the smoothness of $h_c$ is automatic), we deduce that there is a disk of radius order $\e$ around $x$ in which $|h_d-h_c|$ is large. Then, the other resistance estimate from Section \ref{sec:resistance} implies that the energy of $h_d - h_c$ must be large. This contradicts Proposition \ref{energy convergence}, so we conclude that there is a uniform bound on $|h_d - h_c|$.

\section{Orthodiagonal representations of planar maps}\label{sec:orthorep}

A planar map is a graph that can be properly embedded in the plane along with a specification for each vertex $v$ of the clockwise order of the edges incident to $v$. This specification determines the faces of the map, see \cite[Ch.~3]{NStFlour18} for details and further background. A {\bf finite triangulation with boundary} is a finite connected planar map in which all faces are triangles except for a distinguished outer face whose boundary is a simple cycle.

We say that an orthodiagonal map $G = (V^\bullet \sqcup V^\circ, E)$ is an {\bf orthodiagonal representation} of a planar map $H$ if $G^\bullet$ and $H$ are isomorphic as planar maps; when $H$ has a distinguished outer face, we require the outer face of $G^\bullet$ to correspond to it under the isomorphism. To apply the results of this paper, it is natural to find conditions on $H$ guaranteeing the existence of an orthodiagonal representation.

For a finite simple triangulation with boundary, it is well-known that an orthodiagonal representation can be constructed using circle packing. A {\bf circle packing} of a simple connected planar map $H$ with vertex set $W$ is a collection $\mcl P = \{C_w\}_{w \in W}$ of circles in the plane with disjoint interiors such that $C_w$ is tangent to $C_{w'}$ if and only if $w$ is adjacent to $w'$ in $H$. The packing $\mcl P$ induces an embedding of $H$ in the plane in which each vertex $w$ is drawn at the center of $C_w$ and adjacent vertices are connected by straight lines; as part of the definition, we require this embedding to respect the planar map structure of $H$, including the outer face if one has been chosen. Koebe's \cite{K36} \emph{circle packing theorem} (see also \cite[Ch.~3]{NStFlour18} and \cite{St05}) states that every finite simple connected planar map has a circle packing.

The proof of the next proposition shows how to build an orthodiagonal representation out of a circle packing of a finite simple triangulation with boundary. As previously mentioned, this is a well-known construction, see \cite[Remark 5]{M01} and \cite[Section 1.2]{W15}.

\begin{prop} \label{prop:orthodiagonal-triangulation}
Let $H$ be a finite simple triangulation with boundary and let $\mcl P$ be a circle packing of $H$. Then there is an orthodiagonal representation of $H$ whose primal graph coincides with the straight-line embedding of $H$ induced by $\mcl P$.
\end{prop}

\begin{proof}
The argument below is illustrated by Figure \ref{circlepack}.

\begin{figure}
\centering
\begin{subfigure}{0.3\textwidth}
  \centering
  \includegraphics[width=1\textwidth]{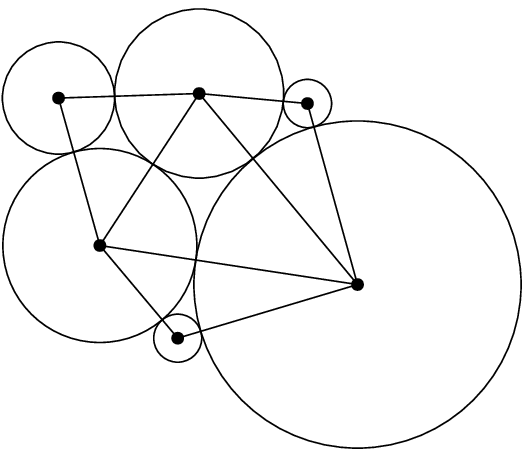}
  \caption{Circle packing of a finite simple triangulation with boundary.}
  \label{circlepack-1}
\end{subfigure}
\quad
\begin{subfigure}{0.3\textwidth}
  \centering
  \includegraphics[width=1\textwidth]{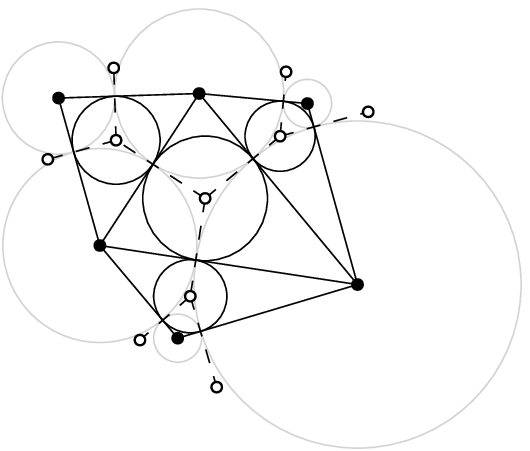}
  \caption{Inscribed circles of the inner faces along with the points $\{c_f\}_{f \in F}$, $\{p_e\}_{e \in B}$.}
  \label{circlepack-2}
\end{subfigure}
\quad
\begin{subfigure}{0.3\textwidth}
  \centering
  \includegraphics[width=1\textwidth]{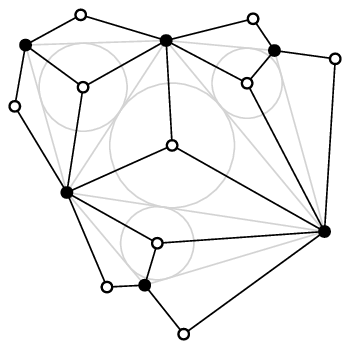}
  \caption{Orthodiagonal representation of the triangulation with boundary.}
  \label{circlepack-3}
\end{subfigure}
\caption{Constructing an orthodiagonal representation of a finite simple triangulation with boundary via circle packing.}
\label{circlepack}	
\end{figure}

Write $\mcl P = \{C_w\}_{w \in W}$, where $W$ is the vertex set of $H$. Let $F$ denote the set of inner faces (which are all triangles) of the straight-line embedding of $H$ induced by $\mcl P$. For each $f \in F$, let $C_f$ be the inscribed circle of the triangle $f$. Denote the centers of the circles $C_w, C_f$ by $c_w, c_f$.

The key geometric fact underlying the construction is this. If $e$ is an edge of a face $f \in F$, with endpoints $c_w, c_{w'}$ for $w,w' \in W$, then the tangency point of $C_f$ with $e$ is the same as the tangency point of $C_w$ with $C_{w'}$. We label this point $q_e$. It follows that if $e$ is incident to two faces $f,f' \in F$, then the line segment $c_f c_{f'}$ passes through $q_e$ and is orthogonal to $e$. Hence the quadrilateral $c_w c_f c_{w'} c_{f'}$ has orthogonal diagonals.

We must do a little extra work at the boundary. Let $B$ be the set of edges in the straight-line embedding of $H$ that are incident to the outer face. Each $e \in B$ is also an edge of an inner face $f \in F$, and the line segment $c_f q_e$ is orthogonal to $e$. We extend this segment a short distance past $q_e$ into the outer face and label the new endpoint $p_e$. If $c_w, c_{w'}$ are the endpoints of $e$, then the quadrilateral $c_w c_f c_{w'} p_e$ has orthogonal diagonals. By doing this, we have carved out a triangular region $c_w p_e c_{w'}$ from the outer face; we make the extensions short enough that the triangular regions associated with different edges in $B$ are pairwise disjoint.

We can now define the orthodiagonal representation $G = (V^\bullet \sqcup V^\circ, E)$ of $H$ by
\[
V^\bullet = \{ c_w : w \in W \}\,, \qquad V^\circ = \{ c_f : f \in F \} \cup \{ p_e : e \in B \}
\]
and $E = E_1 \cup E_2$ where
\begin{align*}
E_1 &= \{ c_w c_f : w \in W,\, f \in F,\, \text{$c_w$ is a vertex of $f$} \} \\
E_2 &= \{ c_w p_e : w \in W,\, e \in B,\, \text{$c_w$ is an endpoint of $e$} \} \, .
\end{align*}
The inner faces of $G$ are quadrilaterals of the form $c_w c_f c_{w'} c_{f'}$ or $c_w c_f c_{w'} p_e$, which have orthogonal diagonals as discussed above. The other required properties are easy to check.
\end{proof}

Every finite simple triangulation with boundary can be circle packed so that the circles corresponding to vertices of the outer face are internally tangent to the unit circle $\partial \DD = \{z : |z|=1\}$ and all other circles are contained in the unit disk $\DD = \{z : |z| < 1\}$ \cite[Claim~4.9]{NStFlour18}. We call this a ``circle packing in $\DD$.''
Given such a packing, the orthodiagonal representation in Proposition \ref{prop:orthodiagonal-triangulation} is determined except for the locations of the extra vertices $p_e$, which may be placed arbitrarily close to their corresponding boundary edges $e$. The following corollary provides conditions under which Theorem \ref{main_thm} may profitably be applied with $\Omega = \DD$.

\begin{cor} \label{cor:D-approx}
Consider a circle packing in $\DD$ of a finite simple triangulation with boundary. Assume that one of the circles in the packing is centered at the origin. Then the associated orthodiagonal map in Proposition \ref{prop:orthodiagonal-triangulation}, which we label $G$, can be drawn such that:
\begin{enumerate}[label=(\roman*)]
\item If the maximum radius among all the circles in the packing is at most $\e$, then the maximal edge length of $G$ is at most $2\e$.
\item If the maximum radius among the circles internally tangent to $\partial \DD$ is at most $\delta$, then the Hausdorff distance between $\partial G$ and $\partial \DD$ is at most $2\delta$.
\end{enumerate}
\end{cor}

The condition that one of the circles is centered at the origin can be satisfied by applying a M\"{o}bius transformation. Alternatively, the conclusion of Corollary \ref{cor:D-approx} still holds (with the same proof) as long as one of the circle centers is contained in the open disk $\{z : |z| < 1 - 2\delta\}$.

\begin{proof}
We use the notation from the proof of Proposition \ref{prop:orthodiagonal-triangulation}. We place the vertices $p_e$ so that they are all inside $\DD$ and so that each length $|q_e p_e| \leq \e$.

Given an edge $c_w c_f$ of $G$, let $e$ be an edge of $f$ that is incident to $c_w$. Then the segment $c_w q_e$ is a radius of the circle $C_w$, and the segment $c_f q_e$ is a radius of $C_f$. Since the circle $C_f$ is inscribed in a triangle whose side lengths are at most $2\e$, its radius is no more than $\e$ (in fact it is bounded by $\e / \sqrt{3}$). Hence $|c_w c_f| \leq |c_w q_e| + |q_e c_f| \leq \e + \e$. Given an edge $c_w p_e$ of $G$, we have $|c_w p_e| \leq |c_w q_e| + |q_e p_e| \leq \e + \e$. This verifies (i).

For (ii), let $S$ be the union of the closures of the inner faces $f \in F$. Thus $S$ is a closed set whose boundary $\partial S$ is the union of the edges $e \in B$. Each $e \in B$ is a segment $c_w c_{w'}$ of length at most $2\delta$, where both $|c_w|, |c_{w'}| \geq 1-\delta$. Hence every point $z \in \partial S$ satisfies $|z| \geq 1-2\delta$. As $S$ contains the origin, it must contain the entire disk $\{z : |z| \leq 1-2\delta \}$.

The set $\partial G$ is the union of the edges $c_w p_e$ of $G$. These edges lie outside of $S$ except for the endpoints $c_w$, which are in $\partial S$. In addition, each edge $c_w p_e$ lies inside $\DD$ because both of its endpoints are in $\DD$. Hence $\partial G \subset \{z : 1-2\delta \leq |z| < 1 \}$ and in particular, each point of $\partial G$ is within distance $2\delta$ of $\partial \DD$. Conversely, for any $u \in \partial \DD$ we may draw the line segment from $u$ to the origin. Since the origin is contained in $S \subset \widehat{G}$ while $u \notin \widehat{G}$, this segment must intersect $\partial G$ at a point $z$ with $|u-z| = 1 - |z| \leq 2\delta$.
\end{proof}

We now introduce the \emph{double circle packing theorem} and show how it can be employed to obtain orthodiagonal representations of planar maps that are not necessarily triangulations. We were not able to find this observation in the literature (although the method is essentially the same as the one in Proposition \ref{prop:orthodiagonal-triangulation}), so we provide a formal statement and quick proof here. The construction works for finite simple planar maps that are 3-connected.

The double circle packing theorem follows from Thurston's interpretation of Andreev's theorem (see \cite[Ch.~13]{Thurston}, \cite{MR90}) and was also proved by Brightwell and Scheinerman \cite{BriSch93}. It is easiest to state using circle packings on the sphere $\C \cup \{\infty\}$. In those terms, it says the following. Let $H$ be a finite simple 3-connected planar map with vertex set $W$ and face set $F$. (Unlike in the case of a triangulation with boundary, we do not distinguish an outer face.) Then there are two collections of circles on the sphere, $\mcl P = \{C_w\}_{w \in W}$ and $\mcl P^\dagger = \{C_f\}_{f \in F}$, such that:
\begin{itemize}
\item The collection $\mcl P$ is a circle packing of $H$.
\item The circles in $\mcl P^\dagger$ are internally disjoint, and two circles $C_f, C_{f'}$ are tangent if and only if the faces $f,f'$ share an edge of $H$.
\item Given an edge of $H$ that is incident to the vertices $w,w' \in W$ and the faces $f,f' \in F$, the point of tangency between $C_w$ and $C_{w'}$ is the same as the point of tangency between $C_f$ and $C_{f'}$. At this point, the circles $C_w, C_{w'}$ are orthogonal to the circles $C_f, C_{f'}$.
\end{itemize}
This construction is called a double circle packing of $H$ on the sphere and is unique up to M\"{o}bius transformations. If we place $\infty$ at a point outside all of the circles in $\mcl P$, then the stereographic projection of $(\mcl P, \mcl P^\dagger)$ is called a double circle packing of $H$ in the plane.

\begin{thm} \label{thm:orthodiagonal-representation}
Let $H$ be a finite simple 3-connected planar map and let $(\mcl P, \mcl P^\dagger)$ be a double circle packing of $H$ in the plane. Then there is an orthodiagonal representation of $H$ whose primal graph coincides with the straight-line embedding of $H$ induced by $\mcl P$.
\end{thm}

\begin{proof}
The argument below is illustrated by Figure \ref{doublepack}.

\begin{figure}
\centering
\begin{subfigure}{0.3\textwidth}
  \centering
  \includegraphics[width=0.8\textwidth]{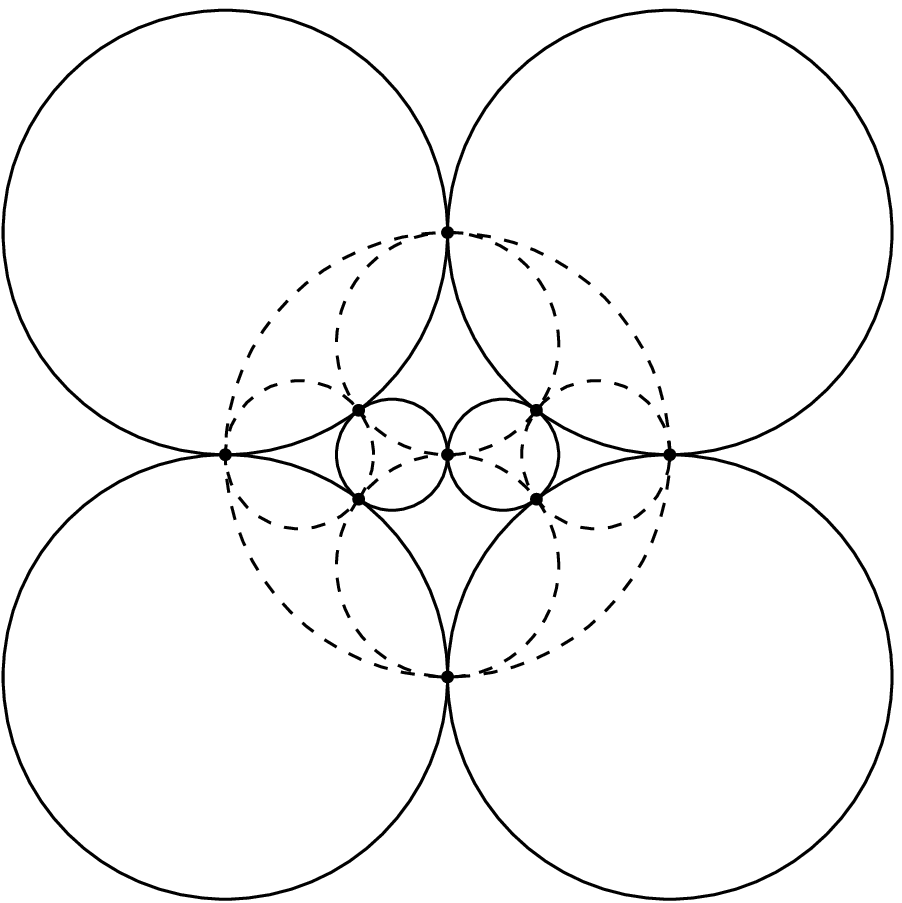}
  \caption{Double circle packing of a $3$-connected planar map.}
  \label{doublepack-2}
\end{subfigure}
\quad
\begin{subfigure}{0.3\textwidth}
  \centering
  \includegraphics[width=0.8\textwidth]{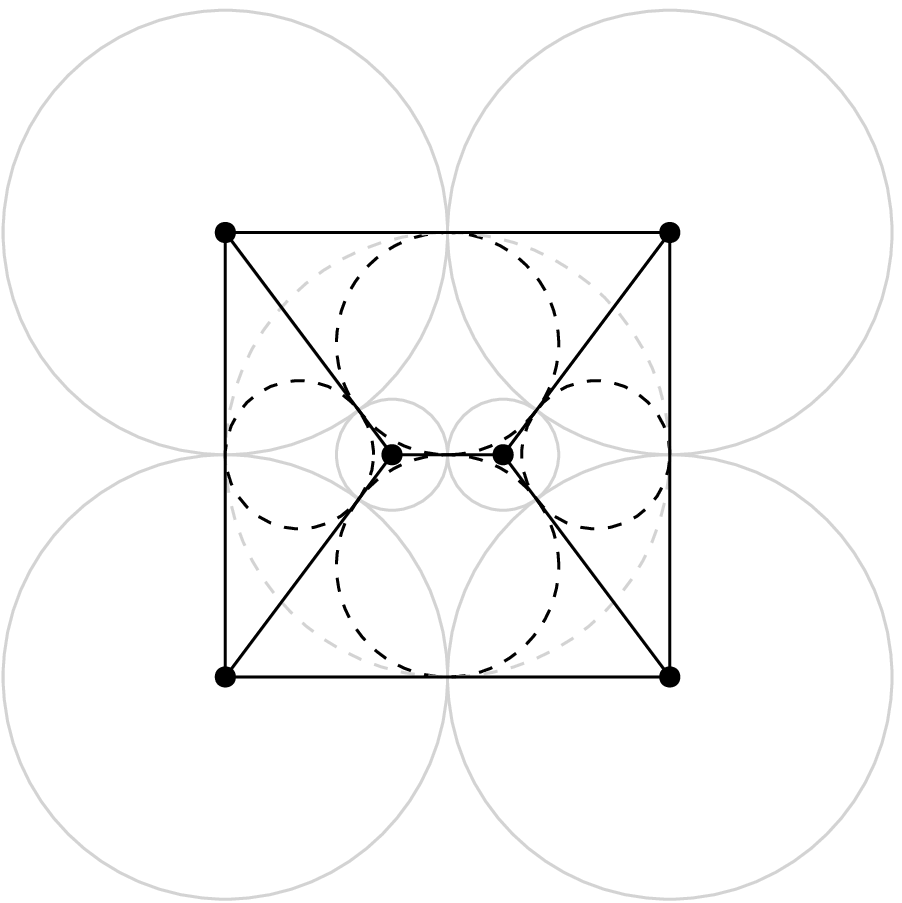}
  \caption{Circles associated with inner faces are inscribed in those faces.}
  \label{doublepack-e}
\end{subfigure}
\quad
\begin{subfigure}{0.3\textwidth}
  \centering
  \includegraphics[width=0.8\textwidth]{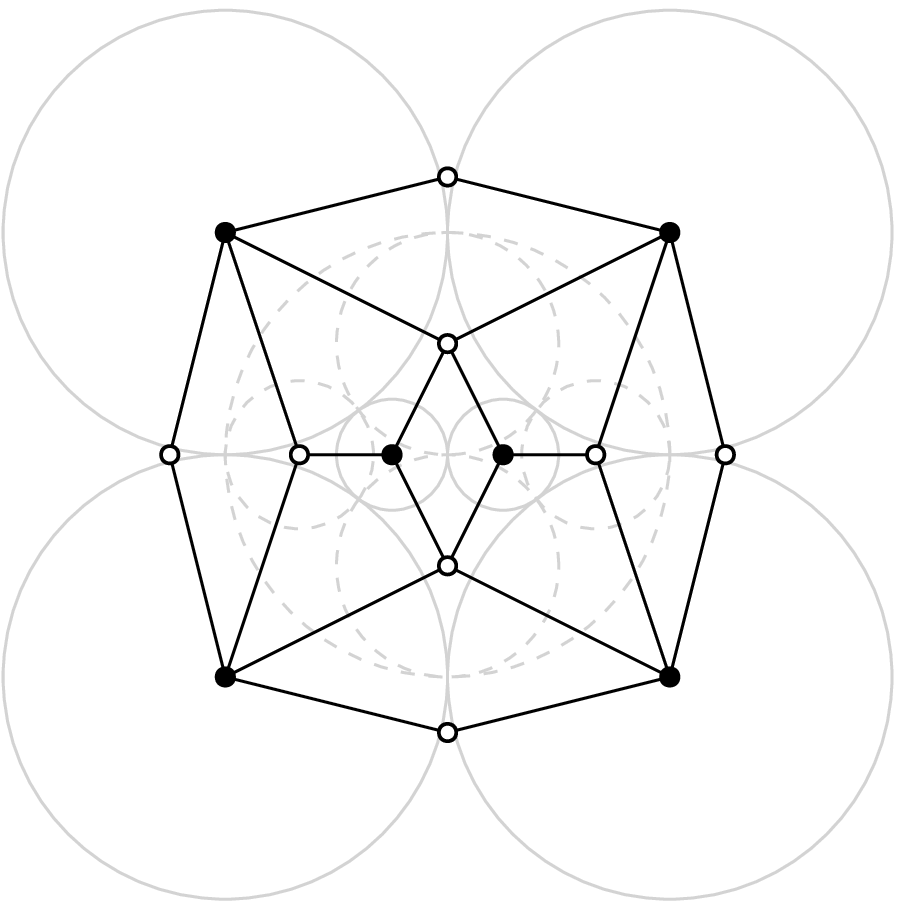}
  \caption{Orthodiagonal representation of the planar map.}
  \label{doublepack-3}
\end{subfigure}
\caption{Using the double circle packing of a finite simple $3$-connected planar map to find an orthodiagonal representation.}
\label{doublepack}	
\end{figure}

Write $\mcl P = \{C_w\}_{w \in W}$ and $\mcl P^\dagger = \{C_f\}_{f \in F}$, where $W$ and $F$ are respectively the vertex set and face set of $H$. We identify each $f \in F$ with its image under the embedding. Since each face is bounded by a simple cycle \cite[Proposition 4.2.5]{D17}, all the inner faces are polygons. Suppose that $f$ is an inner face and $e$ is an edge of $f$. Let $C_w, C_{w'}$ be the circles in $\mcl P$ centered at the endpoints of $e$. As in the proof of Proposition \ref{prop:orthodiagonal-triangulation}, let $q_e$ be the point along $e$ at which $C_w$ and $C_{w'}$ are tangent. Then $e$ is orthogonal to $C_w$ and $C_{w'}$ at $q_e$. By the definition of double circle packing, the circle $C_f$ also passes through $q_e$ at an angle orthogonal to $C_w$ and $C_{w'}$, so $C_f$ is tangent to $e$ at $q_e$. Because $C_f$ is tangent to all the edges of the polygon $f$, it must be inscribed in $f$.

Our situation is now the same as in the proof of Proposition \ref{prop:orthodiagonal-triangulation}: each inner face $f$ has an inscribed circle $C_f$ which is tangent to the edges $e$ of $f$ at the points $q_e$. Therefore, an orthodiagonal representation of $H$ can be constructed in exactly the same way.
\end{proof}

Finally, we provide an analogue of Corollary \ref{cor:D-approx} for double circle packings.

\begin{cor} \label{cor:D-approx-double}
Consider a double circle packing in the plane $(\mcl P, \mcl P^\dagger)$ of a finite simple $3$-connected planar map. Let $\fout$ be the outer face in the straight-line embedding of the map induced by $\mcl P$ and assume that its corresponding circle $C_{\fout} \in \mcl P^\dagger$ is the unit circle $\partial \DD$. Then the associated orthodiagonal map in Theorem \ref{thm:orthodiagonal-representation}, which we label $G$, can be drawn such that:
\begin{enumerate}[label=(\roman*)]
\item If the maximum radius among all the circles in $\mcl P \cup \mcl P^\dagger$ other than $C_{\fout}$ is at most $\e$, then the maximal edge length of $G$ is at most $2\e$.
\item If the maximum radius among the circles $C_w \in \mcl P$ for vertices $w$ on the boundary of $\fout$ is at most $\delta$, then the Hausdorff distance between $\partial G$ and $\partial \DD$ is at most $\delta$.
\end{enumerate}
\end{cor}

We note that in Corollary \ref{cor:D-approx-double}, the outer circles in $\mcl P$ (the ones corresponding to vertices of the outer face) intersect $\partial \DD$ orthogonally. The circles in $\mcl P^\dagger \setminus \{C_{\fout}\}$, not those in $\mcl P$, are packed in $\DD$ in the sense of Corollary \ref{cor:D-approx}. To further compare the statements of the two corollaries, suppose in the setting of Corollary \ref{cor:D-approx-double} that the planar map whose packing is $\mcl P^\dagger \setminus \{C_{\fout}\}$ is a finite simple triangulation with boundary. Condition (i) of Corollary \ref{cor:D-approx-double} bounds the radii of all the circles in $\mcl P^\dagger \setminus \{C_{\fout}\}$ and in $\mcl P$. Condition (i) of Corollary \ref{cor:D-approx} bounds the radii of the circles in $\mcl P^\dagger \setminus \{C_{\fout}\}$; since the map is a triangulation, this controls the radii of all but the outer circles in $\mcl P$. In place of a bound on the radii of the outer circles, Corollary \ref{cor:D-approx} substitutes the assumption that a circle in $\mcl P^\dagger \setminus \{C_{\fout}\}$ is centered at the origin.

\begin{proof}
We use the notation from the proofs of Proposition \ref{prop:orthodiagonal-triangulation} and Theorem \ref{thm:orthodiagonal-representation}. We place the vertices $p_e$ so that each length $|q_e p_e| \leq \delta \leq \e$.

Given an edge $c_w c_f$ of $G$, let $e$ be an edge of $f$ that is incident to $c_w$. Then the segment $c_w q_e$ is a radius of the circle $C_w$, and the segment $c_f q_e$ is a radius of $C_f$. Hence $|c_w c_f| \leq |c_w q_e| + |q_e c_f| \leq \e + \e$. Given an edge $c_w p_e$ of $G$, we have $|c_w p_e| \leq |c_w q_e| + |q_e p_e| \leq \e + \e$. This verifies (i).

For (ii), the set $\partial G$ is the union of the edges $c_w p_e$ of $G$. For each such edge, the point $q_e$ is on the circle $C_{\fout} = \partial \DD$ and we have $|c_w q_e|,|p_e q_e| \leq \delta$. It follows that $|z q_e| \leq \delta$ for every $z$ on the segment $c_w p_e$. Conversely, given a vertex $w \in W$ we write $D_w$ for the closed disk enclosed by $C_w$. We observe that $\partial \DD$ is contained inside the union of the disks $D_w$ for vertices $w$ on the boundary of $\fout$. Therefore each point $z \in \partial \DD$ is in one of these disks and has distance at most $\delta$ from its center $c_w$, which is a point on $\partial G$.
\end{proof}

\section{Preliminaries}
\label{preliminaries}

This chapter collects definitions and results that will be used throughout the rest of the paper. In Section \ref{Plane graphs} we provide some basic facts about plane graphs and networks. Section \ref{Orthodiagonal maps} proves some useful statements about orthodiagonal maps. Finally, Section \ref{Continuous harmonic functions} discusses the classical Dirichlet problem and the properties of its solutions.

\subsection{Graphs, plane graphs and duality}
\label{Plane graphs}

For further background on the definitions below, see \cite{D17,MT01}.

By a graph we will always mean an undirected graph $G = (V,E)$, possibly with loops and multiple edges. All graphs in this paper will be finite. The set of directed edges obtained by choosing both possible orientations of each edge in $E$ is denoted by $\vec{E}$. The tail and head vertices of a directed edge $e \in \vec{E}$ are respectively labeled $e^-, e^+ \in V$. We write $-e$ for the reversed edge. When there is only one edge between $x,y \in V$, we sometimes write $(x,y)$ and $(y,x)$ for the directed edges. There is a natural map from $\vec{E}$ to $E$ that forgets the orientation. By composing with this map, we can and will view any function on $E$ as a function on $\vec{E}$ that assigns the same value to each pair $e,-e \in \vec{E}$.

The following lemma justifies the block decomposition of orthodiagonal maps. It will also be used in Section \ref{main proof}.

\begin{lemma}
\label{block faces}
Let $G$ be a finite plane graph whose inner faces are all bounded by simple closed curves. If $H$ is a block of $G$, then the inner faces of $H$ are all inner faces of $G$.
\end{lemma}

\begin{proof}
Let $f$ be an inner face of $H$ and let $e$ be an (undirected) edge of $H$ that is part of the boundary of $f$. There is a face $f'$ of $G$ that has nonempty intersection with $f$ and whose boundary also includes $e$. Since $H$ is a subgraph of $G$, we have $f' \subseteq f$. Thus $f'$ cannot be the outer face of $G$. The boundary $H'$ of $f'$, viewed as a subgraph of $G$, is a simple cycle. If $e$ is a loop, then it is the only edge in $H'$, so $H' \subseteq H$. Otherwise, the intersection of the $2$-connected graphs $H$ and $H'$ contains two distinct vertices (the endpoints of $e$), so the union $H \cup H'$ is also $2$-connected. By maximality of the blocks, $H' \subseteq H$. Hence $f'$ is a face of $H$, and we conclude that $f = f'$.
\end{proof}

Let $G$ be a finite connected plane graph, and let $F$ be its set of faces. The plane graph $G^\dagger = (V^\dagger, E^\dagger)$, with set of faces $F^\dagger$, is a \textbf{plane dual} of $G$ if:
\begin{enumerate}[label=(\roman*)]
\item There is a bijection between $F$ and $V^\dagger$ such that each $x^\dagger \in V^\dagger$ is contained in the face of $G$ to which it corresponds.
\item There is a bijection between $V$ and $F^\dagger$ such that each $x \in V$ is contained in the face of $G^\dagger$ to which it corresponds.
\item There is a bijection between $E$ and $E^\dagger$ such that, if $e \in E$ is incident to the vertices $x_1,x_2 \in V$ and borders the faces $f_1,f_2 \in F$, then the corresponding dual edge $e^\dagger \in E^\dagger$ is incident to the vertices in $V^\dagger$ that correspond to $f_1,f_2$ via the bijection in (i), and borders the faces in $F^\dagger$ that correspond to $x_1,x_2$ via the bijection in (ii).
\item Each pair of corresponding edges $e \in E$ and $e^\dagger \in E^\dagger$ intersects in exactly one point, and these are the only intersections of $G$ with $G^\dagger$.
\end{enumerate}

It is well-known \cite{D17} that every finite connected plane graph $G$ has a plane dual $G^\dagger$, and that in turn, $G$ is a plane dual of $G^\dagger$.

\subsection{Properties of orthodiagonal maps}\label{Orthodiagonal maps}
In this subsection we prove three fundamental statements about orthodiagonal maps. Lemma \ref{lem:connected} will allow us to apply results that require the graph to be connected, such as those in Section \ref{Electric networks}, to the primal and dual graphs. Lemma \ref{orientation} is a simple orientation property that underlies the proofs of Propositions \ref{martingale} and \ref{energy convergence}. Lastly, Proposition \ref{martingale} says that the edge weights \eqref{conductances} make the random walk on the primal network into a martingale on the interior vertices.

\begin{lemma}
\label{lem:connected}
Given a finite orthodiagonal map $G = (V^\bullet \sqcup V^\circ, E)$, both the primal graph $G^\bullet$ and the dual graph $G^\circ$ are connected.	
\end{lemma}

\begin{proof}
We only show that $G^\bullet$ is connected, since the proof for $G^\circ$ is identical. The argument below is illustrated by Figure \ref{paths}.

\begin{figure}
\centering
\includegraphics[width=0.75\textwidth]{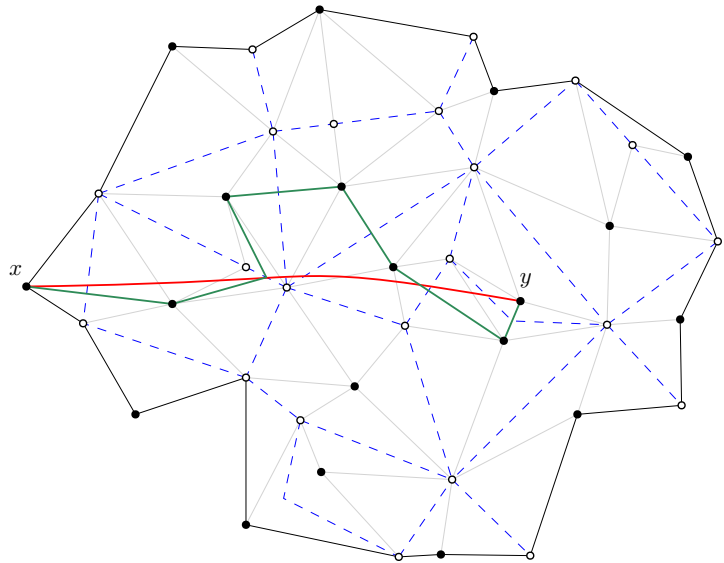}
\caption{Converting a continuous path into a path in $G^\bullet$. The edges of $G$ are in gray except for those in $\partial G$, which are black. The dual graph $G^\circ$ is shown with dashed blue edges. Each connected component of $\widehat{G} \setminus (V^\circ \cup E^\circ)$ contains exactly one primal vertex. The continuous red path from $x$ to $y$ is converted to the green path in $G^\bullet$.}
\label{paths}	
\end{figure}

Let $x,y \in V^\bullet$. Because $\partial G$ is a simple closed curve, we may draw a continuous path from $x$ to $y$ that is entirely contained in $\Int(G)$ except possibly for the two endpoints. Applying a small perturbation if necessary, we may assume that the path does not intersect any vertex in $V^\circ$ and has finitely many intersections with the set $E^\circ$.

The connected components of $\widehat{G} \setminus (V^\circ \cup E^\circ)$ come in two types. First, each $v \in \Int(V^\bullet)$ is contained in a face $f^\circ(v)$ of $G^\circ$ which is a connected component of $\widehat{G} \setminus (V^\circ \cup E^\circ)$. The boundary of $f^\circ(v)$ consists of precisely those dual edges in $E^\circ$ whose corresponding primal edges in $E^\bullet$ are incident to $v$. Second, each $w \in \partial V^\bullet$ is contained in a connected component of $\widehat{G} \setminus (V^\circ \cup E^\circ)$ which we call $f^\circ(w)$ even though it is not a face of $G^\circ$. The boundary of $f^\circ(w)$ consists of the dual edges in $E^\circ$ whose corresponding primal edges in $E^\bullet$ are incident to $w$, along with the two edges of $G$ that are part of $\partial G$ and incident to $w$. In this way, the connected components of $\widehat{G} \setminus (V^\circ \cup E^\circ)$ are in bijection with the vertices in $V^\bullet$.

Write the sequence of connected components of $\widehat{G} \setminus (V^\circ \cup E^\circ)$ traversed in order by the continuous path from $x$ to $y$ as $f^\circ(x_0), f^\circ(x_1), \ldots, f^\circ(x_k)$. We have $x_0 = x$ and $x_k = y$. When the continuous path goes from $f^\circ(x_j)$ to $f^\circ(x_{j+1})$, it must do so by crossing a dual edge in $E^\circ$ whose corresponding primal edge in $E^\bullet$ has endpoints $x_j$ and $x_{j+1}$. Therefore, $(x_0,\ldots,x_k)$ is a path in $G^\bullet$ from $x$ to $y$.
\end{proof}

The following orientation lemma is essentially trivial for convex inner faces and still holds in the general setting. Recall that the notation $Q = [v_1,w_1,v_2,w_2]$ means that the counterclockwise-oriented boundary of $Q$ visits those vertices in order.

\begin{lemma} \label{orientation}
Let $Q = [v_1,w_1,v_2,w_2]$ be an inner face of an orthodiagonal map. Let $\vv$ and $\ww$ be unit vectors pointing in the directions of $\overrightarrow{v_1 v_2}$ and $\overrightarrow{w_1 w_2}$, respectively. Then $\ww$ is the counterclockwise rotation of $\vv$ about the origin by the angle $\pi/2$.
\end{lemma}

\begin{proof}
Since $\vv$ and $\ww$ are orthogonal, the only question is whether the rotation is clockwise or counterclockwise. If $Q$ is convex, then the rotation must be counterclockwise since $w_1$ is on the right side of the directed segment $\overrightarrow{v_1 v_2}$ and $w_2$ is on the left side. If $Q$ is not convex, it can be deformed into a convex quadrilateral with orthogonal diagonals in a way that preserves $\vv$ and $\ww$. See Figure \ref{deform}.
\end{proof}

\begin{figure}
\centering
\includegraphics[width=0.15\textwidth]{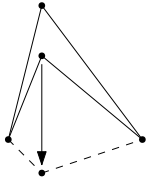}
\caption{Deforming a concave quadrilateral face into a convex quadrilateral without changing the orientation of the diagonals.}
\label{deform}	
\end{figure}

The following proposition, due to Duffin \cite{D68} and Dubejko \cite{D97}, states that the random walk on $(G^\bullet, c)$, which follows each edge with probability proportional to its weight as defined in \eqref{conductances}, is a martingale on the interior vertices $\Int(V^\bullet)$.

\begin{prop}[\cite{D68}, Theorem 2; \cite{D97}, Lemma 3.3] \label{martingale}
Let $(G^\bullet, c)$ be the primal network associated with a finite orthodiagonal map $G = (V^\bullet \sqcup V^\circ, E)$. The random walk $(X_n)$ on $(G^\bullet, c)$ is a martingale on the interior vertices $\Int(V^\bullet)$. That is, the horizontal and vertical coordinate functions are discrete harmonic on $\Int(V^\bullet)$.
\end{prop}

\begin{proof}
Let $v \in \Int(V^\bullet)$ and let $Q_1,\ldots,Q_k$ be the faces of $G$ incident to $v$, listed in counterclockwise order. Taking indices mod $k$, we may write $Q_j = [v,w_j,v_j,w_{j+1}]$, where $w_1,\ldots,w_k \in V^\circ$ are the neighbors of $v$ in $G$ and $v_1,\ldots,v_k \in V^\bullet$ are the neighbors of $v$ in $G^\bullet$. See Figure \ref{local-nbd}. Let $\Exp_v$ denote the expectation for the random walk started at $X_0 = v$. We have
\[
\Exp_v[X_1] - v = \sum_{j=1}^k \frac{c(e^\bullet_{Q_j})}{\pi(v)} (v_j - v) = \frac{1}{\pi(v)} \sum_{j=1}^k |w_j w_{j+1}| \frac{v_j - v}{|v v_j|} \, .
\]

\begin{figure}
\centering
\includegraphics[width=0.55\textwidth]{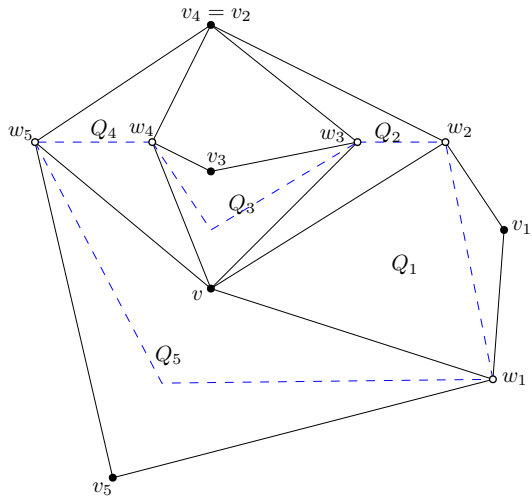}
\caption{Notation for the vertices, edges, and faces in the immediate neighborhood of a primal vertex $v \in \Int(V^\bullet)$. The dual vertices $w_j$ are always distinct, but the primal vertices $v_j$ need not be. The boundary of the face of $G^\circ$ containing $v$ is a simple cycle marked with a blue dashed line.}
\label{local-nbd}	
\end{figure}

Our goal is to show that this sum is zero. Let $T: \R^2 \to \R^2$ be clockwise rotation about the origin by the angle $\pi/2$. By Lemma \ref{orientation},
\[
T\left( \frac{w_{j+1} - w_j}{|w_j w_{j+1}|} \right) = \frac{v_j - v}{|v v_j|} \, .
\]
Therefore,
\[
\Exp_v[X_1] - v = \frac{1}{\pi(v)} \sum_{j=1}^k T(w_{j+1} - w_j) = \frac{1}{\pi(v)} T\left( \sum_{j=1}^k w_{j+1} - w_j \right) = \frac{1}{\pi(v)} T(0) = 0 \, ,
\]
which shows that $(X_n)$ is a martingale at $v$.
\end{proof}

\subsection{Continuous harmonic functions}
\label{Continuous harmonic functions}

This subsection collects all the results about continuous harmonic functions and the Dirichlet problem that will be needed in the paper. Propositions \ref{interior derivative estimates} through \ref{energy minimization} will not be used until Chapter \ref{main proof}.

Let $\Omega \subset \R^2$ be a domain. A real-valued function $f \in C^2(\Omega)$ has first-order derivatives $D_1 f, D_2 f$ and second-order derivatives $D_{ij} f$, $i,j \in \{1,2\}$. Recall that $\nabla f$ and $Hf$ are the gradient and Hessian matrix of $f$.
Given $S \subset \Omega$, we define
\[
\|\nabla f\|_{\infty,S} = \sup_{x \in S} |\nabla f(x)|\, , \qquad \|Hf\|_{\infty,S} = \sup_{x \in S} \|Hf(x)\|_2 \, .
\]

The real-valued function $h \in C^2(\Omega)$ is called \textbf{harmonic} on $\Omega$ if $D_{11}h + D_{22}h = 0$ on $\Omega$. We will sometimes refer to such functions as ``continuous harmonic'' to distinguish from discrete harmonic functions. Proofs of the following three propositions can be found in \cite[Ch.~2]{GT01}.

\begin{prop}[Continuous maximum principle] \label{continuous max principle}
Let $\Omega \subset \R^2$ be a domain and $h: \Omega \to \R$ be harmonic on $\Omega$. If there is $x \in \Omega$ such that $h(x) = \sup_\Omega h$, then $h$ is constant on $\Omega$.
\end{prop}

\begin{prop}[Existence and uniqueness of continuous harmonic extensions] \label{continuous extensions}
Let $\Omega \subset \R^2$ be a bounded simply connected domain, and let $g: \partial \Omega \to \R$ be a continuous function. There is a unique function $h: \overline{\Omega} \to \R$ such that $h = g$ on $\partial \Omega$ and $h$ is harmonic on $\Omega$.
\end{prop}

The function $h$ is called the solution to the \textbf{continuous Dirichlet problem} on $\Omega$ with boundary data $g$.

\begin{prop}[Interior derivative estimates] \label{interior derivative estimates}
Let $h$ be harmonic on the domain $\Omega \subset \R^2$. For any $x \in \Omega$,
\[
|\nabla h(x)| \leq \frac{2}{\dist(x, \partial \Omega)} \sup_{y \in \Omega} |h(y)|\, ,
\]
where $\dist$ denotes Euclidean distance. In addition, for any $i,j \in \{1,2\}$,
\[
|D_{ij}h(x)| \leq \frac{16}{\dist(x, \partial \Omega)^2} \sup_{y \in \Omega} |h(y)|\, .
\]
\end{prop}

The next result is a special case of the main theorem in \cite{H88}. It says that on a bounded simply connected domain, a solution to the continuous Dirichlet problem with Lipschitz boundary data must be H\"{o}lder continuous with exponent $1/2$.

\begin{prop}[H\"{o}lder continuity up to the boundary] \label{Holder}
Let $\Omega \subset \R^2$ be a bounded simply connected domain. Let the continuous function $h: \overline{\Omega} \to \R$ be harmonic on $\Omega$ and satisfy $|h(z) - h(w)| \leq L|z-w|$ for all $z,w \in \partial \Omega$. Then there is a universal constant $C < \infty$ such that
\[
|h(x) - h(y)| \leq CL \diam(\Omega)^{1/2} |x-y|^{1/2} \qquad \text{for all $x,y \in \overline{\Omega}$.}
\]
\end{prop}

\begin{proof}
Follows (with $C = 49$) from equation (1.10) in \cite{H88}.
\end{proof}

Finally, we will use that harmonic functions minimize the \textbf{Dirichlet energy} $\int_\Omega |\nabla h|^2$ over the set of all functions with the same boundary values. A form of this statement that holds for any bounded simply connected domain without conditions on the smoothness of the boundary is proved in \cite[Ch.~II, \S 7, Proposition 10]{DL90}.

\begin{prop}[Dirichlet's principle, continuous form] \label{energy minimization}
Let $\Omega \subset \R^2$ be a bounded simply connected domain, and let the continuous function $g: \overline{\Omega} \to \R$ satisfy $g \in C^1(\Omega)$ and $\int_\Omega |\nabla g|^2 < \infty$. Let $h: \overline{\Omega} \to \R$ be harmonic on $\Omega$ and satisfy $h = g$ on $\partial \Omega$. Then,
\[
\int_\Omega |\nabla h|^2 \leq \int_\Omega |\nabla g|^2 \, .
\]
\end{prop}

\section{Electric networks with multiple sources and sinks}
\label{Electric networks}

The theory of electric networks is an important tool to understand the behavior of random walks on graphs. See \cite{DS84} and \cite[Ch.~2]{LP16} for textbook treatments. The perspective taken by \cite{LP16} is linear-algebraic: results such as the discrete Dirichlet's principle and Thomson's principle follow immediately from the definition of certain orthogonal subspaces and projection operators in an inner product space. This approach was originally developed by \cite{BLPS01} and proved quite useful in their study of uniform spanning trees and forests on infinite graphs.

A \textbf{network} is a graph $G = (V,E)$ together with a function $c: E \to (0,\infty)$. We call $c(e)$ the \textbf{conductance} of the edge $e$, and its reciprocal $r(e) = 1/c(e)$ the \textbf{resistance} of $e$. For each $x \in V$, define the \textbf{stationary measure}
\[
\pi(x) = \sum_{e \in \vec{E} \,:\, e^- = x} c(e) \, .
\]
The \textbf{random walk} on the network $(G,c)$ is the discrete time Markov chain $(X_n)$ on the state space $V$ with transition probabilities
\[
P(x,y) = \frac{1}{\pi(x)} \sum_{\substack{e \in \vec{E} \\ e^- = x,\, e^+ = y}} c(e) \, .
\]
This chain satisfies the detailed balance equations $\pi(x) P(x,y) = \pi(y) P(y,x)$.

The function $f: V \to \R$ is called \textbf{discrete harmonic} at $x \in V$ if $\Exp_x[f(X_1)] = f(x)$, where $\Exp_x$ is the expectation for the Markov chain $(X_n)$ started at $X_0 = x$. This condition is equivalent to the equation
\[
f(x) \sum_{e \in \vec{E} \,:\, e^- = x} c(e) = \sum_{e \in \vec{E} \,:\, e^- = x} f(e^+) c(e) \, .
\]
The definition \eqref{def:discreteharmonicortho} is simply the special case of this equation when the conductances are given by \eqref{conductances}.

We say that $f$ is discrete harmonic on $U \subseteq V$ if it is discrete harmonic at each $x \in U$. The following fundamental results can be found in \cite[Section 2.1]{LP16}.

\begin{prop}[Discrete maximum principle]
\label{max principle}
On a finite network, let $h: V \to \R$ be discrete harmonic on $U \subseteq V$. Let $M$ be the maximum value of $h$ on $V$. If $h(x) = M$ for some $x \in U$, then $h(y) = M$ for all $y \in V$ such that there is a path in the network from $x$ to $y$ whose internal vertices are all in $U$.
\end{prop}

In particular, if $G$ is connected and $h$ is discrete harmonic on all of $V$, then $h$ must be a constant function.

\begin{prop}[Existence and uniqueness of discrete harmonic extensions]
\label{discrete harmonic}
On a finite connected network, fix a proper subset $U \subset V$ and a function $g: V \setminus U \to \R$. There is a unique function $h: V \to \R$ such that $h = g$ on $V \setminus U$ and $h$ is discrete harmonic on $U$. If $(X_n)$ is the random walk on the network and $\tau = \min\{n \geq 0 : X_n \in V \setminus U\}$, then $h(x) = \Exp_x[g(X_\tau)]$ for all $x \in V$.
\end{prop}

The function $h$ is called the \textbf{discrete harmonic extension} of $g$, or the solution to the \textbf{discrete Dirichlet problem} on $U$ with boundary data $g$. \\

The rest of this section develops the linear-algebraic framework of \cite{BLPS01,LP16} and applies it to the case of a finite network with specified voltages at an arbitrary subset of vertices. We may think of this subset as the ``boundary'' and its complement as the ``interior.'' By Proposition \ref{discrete harmonic}, there is a unique voltage function that matches the specified values on the boundary and is discrete harmonic on the interior. Its discrete gradient is the current flow associated with the network. (See definitions below.) Our main result, Proposition \ref{sandwich}, can be informally stated as follows. Suppose that a flow on the interior vertices is close, in a suitable sense, to the discrete gradient of some function. Then both the flow and the discrete gradient are close to the current flow induced by the boundary values of the function. Except for Proposition \ref{sandwich}, all results in this section are contained explicitly or implicitly in \cite[Ch.~2]{LP16}.

Our object of study is a finite connected network $(G = (V,E),c)$. Physically, we imagine that each edge $e \in E$ is a wire with electric resistance $r(e) = 1/c(e)$. Given any $f: V \to \R$, its \textbf{discrete gradient} is the function $c\,df : \vec{E} \to \R$ given by $c\,df(e) = c(e)[f(e^+) - f(e^-)]$. This definition is designed to agree with Ohm's Law: the current through an edge $e$ is the product of the conductance $c(e)$ with the voltage difference between the tail and the head of $e$. With this definition, we impose the convention that current travels from vertices of lower voltage to vertices of higher voltage, which is the opposite of the convention used by \cite{LP16}.

Fix a proper subset $U \subset V$ and a function $g: V \setminus U \to \R$. From the electric perspective, the discrete harmonic extension $h: V \to \R$ of $g$ is called the \textbf{voltage function}, and its discrete gradient $c\,dh$ is the \textbf{current flow}, associated with the network and the specified boundary values.

A function $\theta: \vec{E} \to \R$ is \textbf{antisymmetric} if $\theta(-e) = -\theta(e)$ for all $e \in \vec{E}$. The space of all antisymmetric functions from $\vec{E}$ to $\R$ will be denoted by $\ell^2_-(\vec{E})$, with the inner product
\[
(\theta, \varphi)_r = \frac{1}{2} \sum_{e \in \vec{E}} \theta(e) \varphi(e) r(e) \, .
\]
The factor of $\frac{1}{2}$ is because $\theta(e) \varphi(e) r(e) = \theta(-e) \varphi(-e) r(-e)$, so each term in the sum appears twice. The \textbf{energy} of $\theta \in \ell^2_-(\vec{E})$ is
\[
\EE(\theta) = \|\theta\|_r^2 = (\theta, \theta)_r \, .
\]
Since every discrete gradient is antisymmetric, we can also define the energy of a function $f: V \to \R$ by
\[
\EE(f) = \EE(c\,df) = \frac{1}{2} \sum_{e \in \vec{E}} c(e) [f(e^+) - f(e^-)]^2 \, .
\]
This is the discrete analogue of the continuous notion of Dirichlet energy from Section \ref{Continuous harmonic functions}.

Given $e \in \vec{E}$, we write $\chi^e = \1_{\{e\}} - \1_{\{-e\}} \in \ell^2_-(\vec{E})$. Similarly, if $\gamma = (e_1,\ldots,e_k)$ is a directed path in $G$, we write $\chi^\gamma = \chi^{e_1} + \cdots + \chi^{e_k}$. The \textbf{star} at $x \in V$ is
\[
\star_x = \sum_{e^- = x} c(e) \chi^e \, .
\]
We record two useful facts. First, the inner product of any $\theta \in \ell^2_-(\vec{E})$ with $\star_x$ is
\[
(\theta, \star_x)_r = \sum_{e^- = x} \theta(e)\, ,
\]
the net flow out of the vertex $x$. Second is the following lemma.

\begin{lemma}
\label{star derivative}
On a finite network, given a function $f:V \to \R$, we have
\[
c\,df = -\sum_{x \in V} f(x) \star_x\, .
\]	
\end{lemma}

\begin{proof}
For each $x \in V$, we compute directly that $c\,d\1_{\{x\}} = -\star_x$. The result follows by writing $f = \sum_{x \in V} f(x) \1_{\{x\}}$ and using linearity of the discrete gradient.
\end{proof}

Two linear subspaces of $\ell^2_-(\vec{E})$, the star space and the cycle space, are fundamental to the linear-algebraic approach of \cite{BLPS01} and \cite{LP16}. The \textbf{star space} is the subspace of $\ell^2_-(\vec{E})$ spanned by the stars:
\[
\bigstar = \Span \langle \star_x : x \in V \rangle\, .
\]
The \textbf{cycle space} is the subspace of $\ell^2_-(\vec{E})$ spanned by the cycles:
\[
\diamondsuit = \Span \langle \chi^\gamma : \text{$\gamma$ is a directed cycle in $G$} \rangle\, .
\]
Given a subset of vertices $U \subseteq V$, we also define the span of its stars:
\[
\bigstar_U = \Span \langle \star_x : x \in U \rangle\, .
\]

The star and cycle spaces are closely related to two conditions that characterize electric currents. We say that $\theta \in \ell^2_-(\vec{E})$ satisfies Kirchhoff's \textbf{node law} at $x \in V$ if
\[
\sum_{e^- = x} \theta(e) = 0\, ,
\]
or equivalently, if $\theta$ is orthogonal to $\star_x$. Given $U \subseteq V$, a \textbf{flow} on $U$ is a function $\theta \in \ell^2_-(\vec{E})$ that satisfies the node law at each $x \in U$. Thus, $\theta$ is a flow on $U$ if and only if $\theta \in \bigstar_U^\perp$, the orthogonal complement of $\bigstar_U$. For the other condition, we say that $\theta$ satisfies Kirchhoff's \textbf{cycle law} if for every directed cycle $(e_1,\ldots,e_k)$ in $G$,
\[
\sum_{i=1}^k r(e_i) \theta(e_i) = 0\, .
\]
It is immediate that $\theta$ satisfies the cycle law if and only if $\theta \in \diamondsuit^\perp$.

\begin{lemma} \label{cycle law}
On a finite connected network, $\theta \in \ell^2_-(\vec{E})$ satisfies the cycle law if and only if $\theta$ is the discrete gradient of some $f: V \to \R$. Given $\theta$, the function $f$ is unique up to an additive constant.
\end{lemma}

\begin{proof}
If $\theta = c\,df$, then the sum in the definition of the cycle law telescopes to zero. Conversely, assume that $\theta$ satisfies the cycle law and fix $x_0 \in V$. Choose any value for $f(x_0)$. For $x \neq x_0$, let $(e_1,\ldots,e_k)$ be a directed path in $G$ from $x_0$ to $x$ and set
\begin{equation}
\label{cycle construction}
f(x) = f(x_0) + \sum_{i=1}^k r(e_i) \theta(e_i)\, .
\end{equation}
By the cycle law, the value of the sum does not depend on the choice of directed path $(e_1,\ldots,e_k)$. An appropriate choice of paths shows that $c\,df = \theta$. Also, since condition \eqref{cycle construction} is necessary for $c\,df = \theta$, $f$ is unique up to the choice of $f(x_0)$.
\end{proof}

\begin{lemma} \label{node law}
On a finite network, let $f: V \to \R$ and let $\theta = c\,df$. Then, $\theta$ satisfies the node law at $x \in V$ if and only if $f$ is discrete harmonic at $x$.	
\end{lemma}

\begin{proof}
Follows immediately from the definitions.	
\end{proof}

One consequence of Lemma \ref{node law} is that the current flow associated with the discrete harmonic extension of a function $g: V \setminus U \to \R$ is a flow on $U$.

\begin{lemma} \label{star-cycle decomposition}
On a finite connected network, $\ell^2_-(\vec{E}) = \bigstar \oplus \diamondsuit$.	
\end{lemma}

\begin{proof}
A direct computation shows that each star $\star_x$ is orthogonal to each cycle $\chi^\gamma$, so $\bigstar$ and $\diamondsuit$ are orthogonal. If $\theta \in (\bigstar \oplus \diamondsuit)^\perp$, then $\theta$ satisfies the cycle law and the node law at each vertex. Lemma \ref{cycle law} implies that $\theta = c\,df$ for some $f: V \to \R$, and Lemma \ref{node law} implies that $f$ is discrete harmonic on all of $V$. By Proposition \ref{max principle}, $f$ is a constant function, meaning that $\theta \equiv 0$.	
\end{proof}

Given a proper subset $U \subset V$, let $I_U$ be the space of current flows on $U$, that is, discrete gradients of functions $h: V \to \R$ that are discrete harmonic on $U$.

\begin{lemma} \label{custom decomposition}
On a finite connected network, for any proper subset $U \subset V$, $\ell^2_-(\vec{E}) = \bigstar_U \oplus I_U \oplus \diamondsuit$.	
\end{lemma}

\begin{proof}
Since $\bigstar_U$ and $\diamondsuit$ are orthogonal, we show that $I_U = (\bigstar_U \oplus \diamondsuit)^\perp$. Each $\theta \in I_U$ is in $\diamondsuit^\perp$ by Lemma \ref{cycle law} and in $\bigstar_U^\perp$ by Lemma \ref{node law}. The reverse inclusion follows from the same reasoning used to prove Lemma \ref{star-cycle decomposition}.
\end{proof}

Let $\Pi_U: \ell^2_-(\vec{E}) \to \ell^2_-(\vec{E})$ be the orthogonal projection operator onto $I_U$. Much of the theory of electric networks can be understood in terms of properties of $\Pi_U$. The next proposition says that when $\Pi_U$ is applied to a discrete gradient, it preserves boundary values.

\begin{prop} \label{projection properties}
Consider a finite connected network with a proper subset $U \subset V$. Let $f: V \to \R$ and $\theta = c\,df$. Then $\Pi_U \theta = c\,dh$, where $h: V \to \R$ is the unique function that equals $f$ on $V \setminus U$ and is discrete harmonic on $U$. By consequence, $\EE(h) \leq \EE(f)$.
\end{prop}

\begin{proof}
We have $\theta \in \diamondsuit^\perp$ by Lemma \ref{cycle law}, and then $\theta - \Pi_U \theta \in \bigstar_U$ by Lemma \ref{custom decomposition}. Write
\[
\theta - \Pi_U \theta = \sum_{x \in U} a(x) \star_x
\]
for some coefficients $a(x) \in \R$. Extend $a$ to be defined on all of $V$ by setting $a(x) = 0$ for $x \in V \setminus U$. By Lemma \ref{star derivative},
\[
c\,da = -\sum_{x \in V} a(x) \star_x = \Pi_U \theta - \theta\, .
\]
If we let $h = f + a$, then $h = f$ on $V \setminus U$ and
\[
c\,dh = \theta + (\Pi_U \theta - \theta) = \Pi_U \theta\, .
\]
Since $\Pi_U \theta \in I_U$, we know that $\Pi_U \theta = c\,dh'$ for some $h': V \to \R$ that is discrete harmonic on $U$. Lemma \ref{cycle law} implies that $h - h'$ is a constant function, so $h$ is also discrete harmonic on $U$. The uniqueness statement is Proposition \ref{discrete harmonic}, and we have $\EE(h) \leq \EE(f)$ because orthogonal projection decreases the $\ell^2_-(\vec{E})$ norm.
\end{proof}

We can now state and prove the main result of this section, which follows easily from Lemma \ref{custom decomposition} and Proposition \ref{projection properties}.

\begin{prop} \label{sandwich}
Consider a finite connected network with a proper subset $U \subset V$. Let $h: V \to \R$ be discrete harmonic on $U$. Suppose we are given both a flow $\theta$ on $U$ and a function $f: V \to \R$ such that $f = h$ on $V \setminus U$. Then $\EE(c\,df - c\,dh) + \EE(\theta - c\,dh) = \EE(c\,df - \theta)$.	
\end{prop}

\begin{proof}
Let $\varphi = c\,df$. By Proposition \ref{projection properties}, $\Pi_U \varphi = c\,dh$. We write
\[
\varphi - \theta = (\varphi - \Pi_U \varphi) + (\Pi_U \varphi - \theta)\, .
\]
Lemmas \ref{cycle law} and \ref{custom decomposition} imply that $\varphi - \Pi_U \varphi \in \bigstar_U$. As well, since both $\Pi_U \varphi$ and $\theta$ are flows on $U$, $\Pi_U \varphi - \theta \in \bigstar_U^\perp$. Therefore,
\[
\begin{split}
\EE(c\,df - \theta) &= \|\varphi - \theta\|_r^2 = \|\varphi - \Pi_U \varphi\|_r^2 + \|\Pi_U \varphi - \theta\|_r^2 \\
&= \|c\,df - c\,dh\|_r^2 + \|c\,dh - \theta\|_r^2 = \EE(c\,df - c\,dh) + \EE(\theta - c\,dh)\, . \qedhere
\end{split}
\]
\end{proof}

In the rest of this section, we prove the discrete principles of Dirichlet and Thomson. We begin with the following useful computation.

\begin{lemma} \label{flow inner product}
Consider a finite network with a proper subset $U \subset V$. Let $\theta$ be a flow on $U$ and let $f: V \to \R$. Then
\[
(\theta, c\,df)_r = \sum_{x \in V \setminus U} f(x) \sum_{e^+ = x} \theta(e)\, .
\]
\end{lemma}

\begin{proof}
We compute
\[
\begin{split}
(\theta, c\,df)_r &= \frac{1}{2} \sum_{e \in \vec{E}} \theta(e)[f(e^+) - f(e^-)] = \frac{1}{2} \sum_{x \in V} f(x) \left[ \sum_{e^+ = x} \theta(e) - \sum_{e^- = x} \theta(e) \right] \\
&= \sum_{x \in V} f(x) \sum_{e^+ = x} \theta(e)\, .
\end{split}
\]
Since $\theta$ is a flow on $U$, $\sum_{e^+ = x} \theta(e) = 0$ for all $x \in U$. This completes the proof.
\end{proof}

Suppose now that $V \setminus U = A \sqcup B$ where $A,B$ are nonempty and disjoint. We call $\theta \in \ell^2_-(\vec{E})$ a \textbf{flow between $A$ and $B$} if it is a flow on $U$. Its \textbf{strength} is
\[
\strength(\theta) = \sum_{e^- \in A} \theta(e) = \sum_{e^+ \in B} \theta(e)\, .
\]
(Applying the node law at each $x \in U$ shows that the two sums are equal.) We say that $\theta$ is a \textbf{flow from $A$ to $B$} if the strength is positive,\footnote{Some other treatments require that $\sum_{e^- = a} \theta(e)$ and $\sum_{e^+ = b} \theta(e)$ be nonnegative for all $a \in A$ and $b \in B$. We do not.} and a \textbf{unit flow} if the strength is $1$.

Given $f: V \to \R$, define
\[
\gap_{A,B}(f) = \min_{b \in B} f(b) - \max_{a \in A} f(a)\, .
\]
The following inequality contains both Dirichlet's and Thomson's principles as special cases.

\begin{prop} \label{Dirichlet Thomson}
Consider a finite network with nonempty disjoint subsets $A,B$ of $V$. For any flow $\theta$ between $A$ and $B$ and any $f: V \to \R$ with $\gap_{A,B}(f) \geq 0$,
\begin{equation} \label{eq:strength-gap}
\strength(\theta) \gap_{A,B}(f) \leq \EE(\theta)^{1/2} \EE(f)^{1/2}\, .
\end{equation}
\end{prop}

\begin{proof}
Let $\alpha = \max_{a \in A} f(a)$ and $\beta = \min_{b \in B} f(b)$, so that $\gap_{A,B}(f) = \beta - \alpha$. Define $g: V \to \R$ by $g(x) = (f(x) \vee \alpha) \wedge \beta$. Then $g \equiv \alpha$ on $A$ and $g \equiv \beta$ on $B$. We also have $|c\,dg(e)| \leq |c\,df(e)|$ for all $e \in \vec{E}$, so $\EE(g) \leq \EE(f)$.

Lemma \ref{flow inner product} implies that
\[
\begin{split}
(\theta, c\,dg)_r &= \sum_{x \in A} g(x) \sum_{e^+ = x} \theta(e) + \sum_{x \in B} g(x) \sum_{e^+ = x} \theta(e) \\
&= \sum_{x \in A} \alpha \sum_{e^+ = x} \theta(e) + \sum_{x \in B} \beta \sum_{e^+ = x} \theta(e) \\
&= \alpha[-\strength(\theta)] + \beta \strength(\theta) = \strength(\theta) (\beta - \alpha)\, .
\end{split}
\]
Thus, by Cauchy-Schwarz,
\[
\strength(\theta) \gap_{A,B}(f) = (\theta, c\,dg)_r \leq \EE(\theta)^{1/2} \EE(g)^{1/2} \leq \EE(\theta)^{1/2} \EE(f)^{1/2}\, . \qedhere
\]
\end{proof}

In the rest of this paper, we will apply Proposition \ref{Dirichlet Thomson} directly instead of using the usual statements of Dirichlet's and Thomson's principles. Here, we demonstrate for the interested reader how to derive those statements from Proposition \ref{Dirichlet Thomson}. Assume that the network is connected. The inequality \eqref{eq:strength-gap} can be rearranged to
\begin{equation} \label{eq:separation}
\frac{\strength(\theta)}{\EE(\theta)^{1/2}} \leq \frac{\EE(f)^{1/2}}{\gap_{A,B}(f)}
\end{equation}
for any $\theta, f$ with both denominators positive. Let $h: V \to \R$ be identically $0$ on $A$, identically $1$ on $B$, and discrete harmonic on $U$. The random walk characterization of $h$ (see Proposition \ref{discrete harmonic}) implies that $0 \leq h \leq 1$ on $V$ and $\strength(c\,dh) > 0$. Tracing through the proof of Proposition \ref{Dirichlet Thomson} with $f = h$ and $\theta = c\,dh$ shows that equality holds at every step. Hence, squaring \eqref{eq:separation},
\[
\sup_\theta \frac{\strength(\theta)^2}{\EE(\theta)} = \frac{\strength(c\,dh)^2}{\EE(c\,dh)} = \frac{\EE(h)}{\gap_{A,B}(h)^2} = \inf_f \frac{\EE(f)}{\gap_{A,B}(f)^2} \, .
\]
This quantity (which equals $\EE(h)$ since $\gap_{A,B}(h) = 1$) is the \textbf{effective conductance} between $A$ and $B$, and its reciprocal is the \textbf{effective resistance}. The right-hand equality is Dirichlet's principle. The left-hand equality is equivalent to
\[
\inf_\theta \frac{\EE(\theta)}{\strength(\theta)^2} = \EE(\varphi)
\]
where $\varphi = c\,dh / \strength(c\,dh)$ is the \textbf{unit current flow} from $A$ to $B$. This is Thomson's principle.

\section{Energy convergence}
\label{Energy convergence}

Let $G = (V^\bullet \sqcup V^\circ, E)$ be a finite orthodiagonal map with primal and dual networks $(G^\bullet, c)$, $(G^\circ, c)$. We will use the notations $\EE^\bullet$ and $\EE^\circ$ for the energy functionals on these networks. Thus,
\[
\EE^\bullet(f) = \frac{1}{2} \sum_{e \in \vec{E}^\bullet} c(e)[f(e^+) - f(e^-)]^2
\]
for every real-valued function $f$ whose domain contains $V^\bullet$, and $\EE^\circ(f)$ is defined similarly.

Recall that $\partial G$ is the boundary of the outer face of $G$ and that $\widehat{G}$ is the closed subset of $\R^2$ enclosed by $\partial G$. Given $f: \widehat{G} \to \R$, both $\EE^\bullet(f)$ and $\EE^\circ(f)$ are defined. The first result in this section is that when $f$ is sufficiently smooth and the edges of $G$ are short, the average of these two discrete energies approximates the Dirichlet energy $\int_{\widehat{G}} |\nabla f|^2$. For technical reasons, we require $f$ to be smooth on a slightly larger set than $\widehat{G}$. Let
\[
\widetilde{G} = \bigcup_{Q \subset G} \conv(\overline{Q})\, ,
\]
the union over all inner faces $Q$ of the convex hull of the closure of $Q$. Then $\widehat{G} \subseteq \widetilde{G}$, with equality when all inner faces are convex.

\begin{prop} \label{sum integral convergence}
Let $G = (V^\bullet \sqcup V^\circ, E)$ be a finite orthodiagonal map with maximal edge length at most $\e$. Let $U$ be an open subset of $\R^2$ that contains $\widetilde{G}$, and let $f: U \to \R$ be a $C^2$ function. Set
\[
L = \|\nabla f\|_{\infty,\widehat{G}}\, , \qquad M = \|Hf\|_{\infty,\widetilde{G}}\, .
\]
Then,
\[
\left| \frac{\EE^\bullet(f) + \EE^\circ(f)}{2} - \int_{\widehat{G}} |\nabla f|^2 \right| \leq \area(\widehat{G}) (10LM\e + 8M^2 \e^2)\, .
\]
\end{prop}

The second result bounds the discrete energy of the difference between the solution to the continuous Dirichlet problem on $\widehat{G}$ and the solution to the discrete Dirichlet problem on $G^\bullet$ with the same boundary data.

\begin{prop} \label{energy convergence}
Let $G = (V^\bullet \sqcup V^\circ, E)$ be a finite orthodiagonal map with maximal edge length at most $\e$. Let $U$ be a simply connected domain in $\R^2$ that contains $\widetilde{G}$, and let $h_c: U \to \R$ be continuous harmonic on $U$. Let $h_d: V^\bullet \to \R$ be discrete harmonic on $\Int(V^\bullet)$ and satisfy $h_d = h_c$ on $\partial V^\bullet$. Set
\[
M = \|Hh_c\|_{\infty,\widetilde{G}}\, .
\]
Then,
\[
\EE^\bullet(h_c - h_d) \leq 32 \area(\widehat{G}) M^2 \e^2\, .
\]
\end{prop}

The assumption in Proposition \ref{energy convergence} that $h_c$ is harmonic on the larger simply connected domain $U$ could be weakened to the requirement that $h_c$ is harmonic on the interior of $\widehat{G}$ along with some smoothness conditions at the boundary. The stronger assumption will hold when the proposition is applied later in the paper and somewhat simplifies the proof. A similar remark holds for Proposition \ref{sum integral convergence}.

Proposition \ref{sum integral convergence} is analogous to \cite[Lemma 2.3]{S13} and \cite[Lemma 6.1]{W15}, and Proposition \ref{energy convergence} is analogous to \cite[Theorem 3.5]{D99}. The main difference is that all three earlier results impose regularity conditions on $G$, while we only require control over the maximal edge length. It is possible that the bounds above could be strengthened to use a norm other than $L^\infty$ for $L$ and $M$, as in \cite{D99}, but these statements will be sufficient for our purposes.

The rest of this section is devoted to the proofs of these two propositions, which are similar in flavor. Both rely on approximations carried out within each inner face of $G$. While Proposition \ref{sum integral convergence} is effectively just a careful computation, Proposition \ref{energy convergence} relies on the development in Section \ref{Electric networks} of orthogonality in the space $\ell^2_-(\vec{E})$, in particular Proposition \ref{sandwich}.

\begin{lemma} \label{gradient approx}
Let $Q = [v_1,w_1,v_2,w_2]$ be an inner face of an orthodiagonal map $G = (V^\bullet \sqcup V^\circ, E)$. Assume that each edge $v_i w_j$ of $Q$ has length at most $\e$. Let $q \in Q$ be the intersection point of the edges $e_Q^\bullet$ and $e_Q^\circ$. Let $\vv$ and $\ww$ be unit vectors pointing in the directions of $\overrightarrow{v_1 v_2}$ and $\overrightarrow{w_1 w_2}$, respectively. Let $U$ be an open subset of $\R^2$ that contains $\widetilde{Q} = \conv(\overline{Q})$. Given a $C^2$ function $f: U \to \R$, set
\[
M = \|Hf\|_{\infty,\widetilde{Q}}\, .
\]
Then, for each $z \in \overline{Q}$,
\begin{equation} \label{basic-gradient}
|\nabla f(z) - \nabla f(q)| \leq M\e\, .
\end{equation}
In addition,
\begin{align}
\Big| f(v_2) - f(v_1) - \langle \nabla f(q), \vv \rangle |v_1 v_2| \Big| &\leq 2M |v_1 v_2| \e\, , \label{x-gradient} \\
\Big| f(w_2) - f(w_1) - \langle \nabla f(q), \ww \rangle |w_1 w_2| \Big| &\leq 2M |w_1 w_2| \e\, . \label{y-gradient}
\end{align}
\end{lemma}

\begin{proof}
To verify \eqref{basic-gradient}, we observe that the segment $qz$ has length at most $\e$ and is contained in $\overline{Q}$. Let $\uu$ be a unit vector pointing in the direction of $\overrightarrow{qz}$. We compute
\[
\begin{split}
|\nabla f(z) - \nabla f(q)| &= \left| \int_0^{|qz|} \frac{d}{dt} \nabla f(q + t\uu) \,dt \right| = \left| \int_0^{|qz|} Hf(q + t\uu) \uu \,dt \right| \\
&\leq \int_0^{|qz|} \big| Hf(q + t\uu) \uu \big| \,dt \leq M|qz| \leq M\e\, .
\end{split}
\]

For \eqref{x-gradient}, let $\gamma(t) = f(v_1 + t\vv)$, so that
\[
\gamma'(t) = \langle \nabla f(v_1 + t\vv)\, , \vv \rangle \, , \qquad \gamma''(t) = \vv^T Hf(v_1 + t\vv) \vv\, .
\]
In particular, for $0 \leq t \leq |v_1 v_2|$, $v_1 + t\vv \in \widetilde{Q}$ and so $|\gamma''(t)| \leq M$. Now,
\begin{equation} \label{intermediate-gradient}
\Big| f(v_2) - f(v_1) - \gamma'(0) |v_1 v_2| \Big| = \left| \int_0^{|v_1 v_2|} \int_0^t \gamma''(s) \,ds \,dt \right| \leq \frac{M}{2} |v_1 v_2|^2\, .
\end{equation}
We also have
\[
|\gamma'(0) - \langle \nabla f(q), \vv \rangle| = |\langle \nabla f(v_1) - \nabla f(q), \vv \rangle| \leq |\nabla f(v_1) - \nabla f(q)| \leq M\e\, ,
\]
using \eqref{basic-gradient} in the last inequality. Thus,
\[
\Big| \gamma'(0) |v_1 v_2| - \langle \nabla f(q), \vv \rangle |v_1 v_2| \Big| \leq M |v_1 v_2| \e\, .
\]
Combining this bound with \eqref{intermediate-gradient} and using that $|v_1 v_2| \leq 2\e$, we obtain \eqref{x-gradient}. The proof of \eqref{y-gradient} is exactly the same.
\end{proof}

\begin{proof}[Proof of Proposition \ref{sum integral convergence}]
Let $Q = [v_1,w_1,v_2,w_2]$ be an inner face of $G$. We have $\area(Q) = \frac{1}{2} |v_1 v_2| \cdot |w_1 w_2|$. Define $q \in Q$ and the unit vectors $\vv,\ww$ as in the statement of Lemma \ref{gradient approx}. The contribution from $Q$ to $\frac{1}{2}[\EE^\bullet(f) + \EE^\circ(f)]$ is
\begin{equation} \label{energy-from-Q}
\frac{1}{2} \left( \frac{|w_1 w_2|}{|v_1 v_2|} [f(v_2) - f(v_1)]^2 + \frac{|v_1 v_2|}{|w_1 w_2|} [f(w_2) - f(w_1)]^2 \right)\, .
\end{equation}
We will prove that both \eqref{energy-from-Q} and $\int_Q |\nabla f|^2$ are close to $\area(Q) |\nabla f(q)|^2$, and therefore also to each other. Summing over the inner faces $Q$, it will follow that $\frac{1}{2}[\EE^\bullet(f) + \EE^\circ(f)]$ is approximately equal to $\int_{\widehat{G}} |\nabla f|^2$.

To begin, we show that
\begin{multline} \label{half-energy}
\left| \frac{|w_1 w_2|}{|v_1 v_2|} [f(v_2) - f(v_1)]^2 - |v_1 v_2| \cdot |w_1 w_2| \cdot \langle \nabla f(q), \vv \rangle^2 \right| \\
\leq 4 |v_1 v_2| \cdot |w_1 w_2| \cdot (LM\e + M^2 \e^2)\, .
\end{multline}
Indeed, the left side of \eqref{half-energy} is equal to
\[
\frac{|w_1 w_2|}{|v_1 v_2|} \left| [f(v_2) - f(v_1)]^2 - |v_1 v_2|^2 \cdot \langle \nabla f(q), \vv \rangle^2 \right|\, .
\]
In general, if $|a-b| \leq \delta$, then $|a^2 - b^2| \leq |a-b|(|a| + |b|) \leq \delta(2|b| + \delta)$. Using this along with \eqref{x-gradient}, the quantity above is at most
\[
\frac{|w_1 w_2|}{|v_1 v_2|} \cdot 2M |v_1 v_2| \e \cdot \big( 2|v_1 v_2| \cdot |\langle \nabla f(q), \vv \rangle| + 2M |v_1 v_2| \e \big)\, ,
\]
which is bounded above by the right side of \eqref{half-energy}. Similarly, using \eqref{y-gradient}, we also have
\begin{multline} \label{other-half-energy}
\left| \frac{|v_1 v_2|}{|w_1 w_2|} [f(w_2) - f(w_1)]^2 - |v_1 v_2| \cdot |w_1 w_2| \cdot \langle \nabla f(q), \ww \rangle^2 \right| \\
\leq 4 |v_1 v_2| \cdot |w_1 w_2| \cdot (LM\e + M^2 \e^2)\, .
\end{multline}
Since $\vv$ and $\ww$ are orthogonal, $\langle \nabla f(q), \vv \rangle^2 + \langle \nabla f(q), \ww \rangle^2 = |\nabla f(q)|^2$. Therefore, combining \eqref{half-energy} and \eqref{other-half-energy},
\begin{multline} \label{total-energy}
\left| \frac{|w_1 w_2|}{|v_1 v_2|} [f(v_2) - f(v_1)]^2 + \frac{|v_1 v_2|}{|w_1 w_2|} [f(w_2) - f(w_1)]^2 - |v_1 v_2| \cdot |w_1 w_2| \cdot |\nabla f(q)|^2 \right| \\
\leq 8 |v_1 v_2| \cdot |w_1 w_2| \cdot (LM\e + M^2 \e^2)\, .
\end{multline}

We also compute, using \eqref{basic-gradient} in the last line, that
\begin{multline*}
\left| \int_Q |\nabla f(x)|^2 \,dx - \area(Q) |\nabla f(q)|^2 \right| \\
\begin{aligned}
&= \left| \int_Q \Big( |\nabla f(x)|^2 - |\nabla f(q)|^2 \Big) \,dx \right| \\
&\leq \int_Q \Big| \Big( |\nabla f(x)| + |\nabla f(q)| \Big) \Big( |\nabla f(x)| - |\nabla f(q)| \Big) \Big| \,dx \\
&\leq 2L \int_Q \Big| \nabla f(x) - \nabla f(q) \Big| \,dx \\
&\leq 2LM\e \cdot \area(Q)\, .
\end{aligned}
\end{multline*}
Since $\area(Q) = \frac{1}{2} |v_1 v_2| \cdot |w_1 w_2|$, combining the above with $\frac{1}{2} \cdot \eqref{total-energy}$ yields
\begin{multline*}
\left| \frac{1}{2} \left( \frac{|w_1 w_2|}{|v_1 v_2|} [f(v_2) - f(v_1)]^2 + \frac{|v_1 v_2|}{|w_1 w_2|} [f(w_2) - f(w_1)]^2 \right) - \int_Q |\nabla f(x)|^2 \,dx \right| \\
\leq \area(Q) (10LM\e + 8M^2 \e^2)\, .
\end{multline*}
The proof is finished by summing over all inner faces $Q$ of $G$.
\end{proof}

\begin{proof}[Proof of Proposition \ref{energy convergence}]
We define a bijective map $e \mapsto e^\dagger$ from $\vec{E}^\bullet$ to $\vec{E}^\circ$ as follows. Let $Q = [v_1,w_1,v_2,w_2]$ be an inner face of $G$. If $e \in \vec{E}^\bullet$ is the orientation of $e_Q^\bullet$ with $e^- = v_1$ and $e^+ = v_2$, then we set $e^\dagger$ to be the orientation of $e_Q^\circ$ with $(e^\dagger)^- = w_1$ and $(e^\dagger)^+ = w_2$. We also set $(-e)^\dagger = -(e^\dagger)$.

Let $\widetilde{h}_c$ be the harmonic conjugate of $h_c$ on $U$, which is defined up to addition of an arbitrary constant since $U$ is simply connected. Consider the following two functions in $\ell^2_-(\vec{E}^\bullet)$:
\begin{align*}
\varphi(e) &= c(e)[h_c(e^+) - h_c(e^-)] \\
\theta(e) &= \widetilde{h}_c((e^\dagger)^+) - \widetilde{h}_c((e^\dagger)^-)	
\end{align*}
Evidently, $\varphi$ is the discrete gradient of the restriction of $h_c$ to $V^\bullet$. We now show that $\theta$ is a flow on $\Int(V^\bullet)$. Given $v \in \Int(V^\bullet)$, label the vertices and faces in the immediate neighborhood of $v$ as in Figure \ref{local-nbd}. Thus, the faces of $G$ incident to $v$ are listed in counterclockwise order as $Q_1,\ldots,Q_k$, where each $Q_j = [v,w_j,v_j,w_{j+1}]$ (taking indices mod $k$). If $e_j \in \vec{E}^\bullet$ is the edge contained in $Q_j$ with tail $v$ and head $v_j$, then $e_j^\dagger$ has tail $w_j$ and head $w_{j+1}$. It follows that
\[
\sum_{e \in \vec{E}^\bullet \,:\, e^- = v} \theta(e) = \sum_{j=1}^k \left[ \widetilde{h}_c(w_{j+1}) - \widetilde{h}_c(w_j) \right] = 0\, .
\]

Proposition \ref{sandwich}, taking $f$ as the restriction of $h_c$ to $V^\bullet$ and $h = h_d$, implies that
\[
\EE^\bullet(h_c - h_d) \leq \EE^\bullet(h_c - h_d) + \EE^\bullet(\theta - c\,dh_d) = \EE^\bullet(\varphi - \theta)\, .
\]
Thus, it will suffice to show that
\[
\EE^\bullet(\varphi - \theta) \leq 32 \area(\widehat{G}) M^2 \e^2\, .
\]

The contribution to $\EE^\bullet(\varphi - \theta)$ from an inner face $Q = [v_1,w_1,v_2,w_2]$ is
\begin{multline} \label{eq:Q-contrib}
\frac{|v_1 v_2|}{|w_1 w_2|} \left( \frac{|w_1 w_2|}{|v_1 v_2|} [h_c(v_2) - h_c(v_1)] - [\widetilde{h}_c(w_2) - \widetilde{h}_c(w_1)] \right)^2 \\
= |v_1 v_2| \cdot |w_1 w_2| \cdot \left( \frac{h_c(v_2) - h_c(v_1)}{|v_1 v_2|} - \frac{\widetilde{h}_c(w_2) - \widetilde{h}_c(w_1)}{|w_1 w_2|} \right)^2\, .
\end{multline}
Let $\vv$ and $\ww$ be unit vectors pointing in the directions of $\overrightarrow{v_1 v_2}$ and $\overrightarrow{w_1 w_2}$, respectively. By Lemma \ref{orientation}, $\ww$ is the counterclockwise rotation by $\pi/2$ of $\vv$. Therefore, if $q \in Q$ is the intersection point of the edges $e_Q^\bullet$ and $e_Q^\circ$, the Cauchy-Riemann equations imply that
\[
\langle \nabla h_c(q), \vv \rangle = \langle \nabla \widetilde{h}_c(q), \ww \rangle\, .
\]
We rewrite the quantity in parentheses on the right side of \eqref{eq:Q-contrib} as
\[
\frac{h_c(v_2) - h_c(v_1) - \langle \nabla h_c(q), \vv \rangle |v_1 v_2|}{|v_1 v_2|} - \frac{\widetilde{h}_c(w_2) - \widetilde{h}_c(w_1) - \langle \nabla \widetilde{h}_c(q), \ww \rangle |w_1 w_2|}{|w_1 w_2|}\, .
\]
We will apply \eqref{x-gradient} to $h_c$ and \eqref{y-gradient} to $\widetilde{h}_c$. By the Cauchy-Riemann equations,
\[
\|H\widetilde{h}_c\|_{\infty,\widetilde{G}} = \|Hh_c\|_{\infty,\widetilde{G}} = M\, .
\]
Therefore, \eqref{x-gradient} and \eqref{y-gradient} yield
\begin{align*}
\left| \frac{h_c(v_2) - h_c(v_1) - \langle \nabla h_c(q), \vv \rangle |v_1 v_2|}{|v_1 v_2|} \right| &\leq 2M\e\, , \\
\left| \frac{\widetilde{h}_c(w_2) - \widetilde{h}_c(w_1) - \langle \nabla \widetilde{h}_c(q), \ww \rangle |w_1 w_2|}{|w_1 w_2|} \right| &\leq 2M\e\, .
\end{align*}
It follows that the contribution to $\EE^\bullet(\varphi - \theta)$ from $Q$ is at most
\[
|v_1 v_2| \cdot |w_1 w_2| \cdot (2M\e + 2M\e)^2 = 32 \area(Q) M^2 \e^2\, .
\]
Summing over all inner faces $Q$ of $G$, the proof is complete.
\end{proof}

\section{Resistance estimates}
\label{sec:resistance}

The goal of Sections \ref{sec:resistance} through \ref{main proof} is to prove Theorem \ref{main_thm} using the energy bound Proposition \ref{energy convergence}. The two estimates in this section, Propositions \ref{prop:center-to-outside} and \ref{prop:left-to-right}, form the core of the argument. Proposition \ref{prop:left-to-right} will be used in Section \ref{Equicontinuity section} to prove Proposition \ref{equicontinuity}, and Proposition \ref{prop:center-to-outside} will be used in Section \ref{main proof} to help finish the proof of Theorem \ref{main_thm}. In this section we have chosen to put Proposition \ref{prop:center-to-outside} first because its statement is simpler.

\begin{prop} \label{prop:center-to-outside}
Let $G = (V^\bullet \sqcup V^\circ, E)$ be a finite orthodiagonal map with maximal edge length at most $\e$. Fix $x \in V^\bullet$ and let $B$ be the closed disk of radius $r \geq 3\e$ centered at $x$. Assume that $B \subset \Int(G)$. Then there is a unit flow $\theta$ in $G^\bullet$ from the set $A = V^\bullet \cap B$ to $\partial V^\bullet$ such that
\[
\EE^\bullet(\theta) \leq C \log\left( \frac{\diam(\widehat{G})}{r} \right)
\]
for some universal constant $C < \infty$.
\end{prop}

We need one more definition to state Proposition \ref{prop:left-to-right}. Let $G = (V^\bullet \sqcup V^\circ, E)$ be a finite orthodiagonal map. For $\rho > 0$, a {\bf $\bm \rho$-edge} of $G^\bullet$ is an edge in $E^\bullet$ whose corresponding dual edge in $E^\circ$ connects two vertices $w,w' \in V^\circ$ with $|w| < \rho \leq |w'|$.

\begin{prop} \label{prop:left-to-right}
Let $G = (V^\bullet \sqcup V^\circ, E)$ be a finite orthodiagonal map with maximal edge length at most $\e$. Fix $r_1,r_2$ with $r_1 \geq \e$ and $r_2 \geq 2r_1$. Let $S,T$ be disjoint subsets of $V^\bullet$ such that for each $\rho \in (r_1,r_2)$ there is a path in $G^\bullet$ from $S$ to $T$ consisting entirely of $\rho$-edges. Then there is a unit flow $\theta$ in $G^\bullet$ from $S$ to $T$ such that
\[
\EE^\bullet(\theta) \leq \frac{C}{\log(r_2 / r_1)}
\]
for some universal constant $C < \infty$.
\end{prop}

We remark that by Thomson's principle (discussed at the end of Section \ref{Electric networks}), Propositions \ref{prop:center-to-outside} and \ref{prop:left-to-right} provide upper bounds on the effective resistance between $A$ and $\partial V^\bullet$ and between $S$ and $T$. When using the propositions, we will plug the low-energy flows directly into Proposition \ref{Dirichlet Thomson} rather than explicitly considering the effective resistance.

\begin{proof}[Proof of Proposition \ref{prop:center-to-outside}]
Place $x$ at the origin for convenience. We define an antisymmetric function $\varphi$ on $\vec{E}^\bullet$, which will be a flow in $G^\bullet$ from $x$ to $\partial V^\bullet$, as follows. Given a face $Q = [v_1,w_1,v_2,w_2]$ of $G$, let $e \in \vec{E}^\bullet$ be the orientation of $e_Q^\bullet$ with $e^- = v_1$ and $e^+ = v_2$. Choose a branch of $\arg$ which is defined at every point on the edge $e_Q^\circ$ and set $\varphi(e) = \arg(w_2) - \arg(w_1)$, $\varphi(-e) = \arg(w_1) - \arg(w_2)$. Note that the values $\varphi(e),\varphi(-e)$ do not depend on the branch of $\arg$ chosen.

For any $v \in \Int(V^\bullet)$, let $w_1,\ldots,w_k \in V^\circ$ be the neighbors of $v$ in $G$, listed in counterclockwise order around $v$. See Figure \ref{local-nbd}. Taking indices mod $k$,
\begin{equation} \label{telescope}
\sum_{e \in \vec{E}^\bullet \,:\, e^- = v} \varphi(e) = \sum_{j=1}^k [\arg(w_{j+1}) - \arg(w_j)]
\end{equation}
where we may choose the branches of $\arg$ so that the sum on the right side telescopes. The face of $G^\circ$ that contains $v$ is bounded by a simple cycle of edges in $E^\circ$. Let $C_v$ be the counterclockwise orientation of this cycle. Then $C_v$ visits the vertices $w_1,\ldots,w_k$ in order before returning to $w_1$, as shown in Figure \ref{local-nbd}. The right side of \eqref{telescope} is precisely the net change in $\arg$ when traversing $C_v$. This is $2\pi$ multiplied by the winding number of $C_v$ about the origin, where we placed $x$. The only vertex of $V^\bullet$ enclosed by $C_v$ is $v$ itself. Thus, if $v \neq x$, the winding number of $C_v$ about $x$ is zero and $\varphi$ satisfies Kirchhoff's node law at $v$. Since the winding number of $C_x$ about $x$ is $1$, the net flow of $\varphi$ out of $x$ is $2\pi$.

We have shown that $\varphi$ is a flow of strength $2\pi$ from $x$ to $\partial V^\bullet$. In particular, the net flow into $\partial V^\bullet$ is $2\pi$.
Define $\varphi_1: \vec{E}^\bullet \to \R$ by setting $\varphi_1(e) = 0$ if both $e^-, e^+ \in A$ and $\varphi_1(e) = \varphi(e)$ otherwise. For every $v \in V^\bullet \setminus A$, we have $\varphi_1 = \varphi$ on all edges incident to $v$; this means that $\varphi_1$ is a flow of strength $2\pi$ from $A$ to $\partial V^\bullet$. We normalize $\varphi_1$ into a unit flow by setting $\theta = \frac{1}{2\pi} \varphi_1$.

To bound the energy of $\theta$, let $Q = [v_1,w_1,v_2,w_2]$ be a face of $G$ such that $v_1,v_2 \in V_1^\bullet$ are not both in $A$. Let $|\theta(e_Q^\bullet)|$ denote the common value of $|\theta(e)|, |\theta(-e)|$ for both orientations $e,-e$ of $e_Q^\bullet$. The contribution of $e_Q^\bullet$ to $\EE^\bullet(\theta)$ is
\[
\frac{|v_1 v_2|}{|w_1 w_2|} |\theta(e_Q^\bullet)|^2 = \frac{1}{4\pi^2}(|v_1 v_2| \cdot |w_1 w_2|) \left( \frac{
\arg(w_2) - \arg(w_1)}{|w_1 w_2|} \right)^2\, .
\]
We recognize $|v_1 v_2| \cdot |w_1 w_2|$ as twice the area of $Q$.

Let $\widetilde{Q}$ be the convex hull of the closure of $Q$, which has diameter at most $2\e$. Since at least one of $v_1,v_2$ is not in $A$, its distance from $x$ is more than $r$. Thus $\dist(\widetilde{Q}, x) \geq r - 2\e$, and there is a branch of $\arg$ defined on all of $\widetilde{Q}$. Let $\rho_{\mathrm{min}} = \min\{ |w| : w \in w_1 w_2 \}$, so that $\rho_{\mathrm{min}} \geq r - 2\e$ and also $\rho_{\mathrm{min}} \geq |w| - 2\e$ for each $w \in \widetilde{Q}$. Using that $|\nabla \arg(w)| = 1/|w|$,
\[
|\arg(w_2) - \arg(w_1)| \leq |w_1 w_2| / \rho_{\mathrm{min}}\, .
\]
Therefore,
\[
2\area(Q) \left( \frac{\arg(w_2) - \arg(w_1)}{|w_1 w_2|} \right)^2 \leq \frac{2\area(Q)}{\rho_{\mathrm{min}}^2} \leq 2 \int_Q \frac{1}{[(|w| \vee r) - 2\e]^2} \, dw\, .
\]
Let $U$ be the annulus centered at the origin with inner radius $r-2\e$ and outer radius $\diam(\widehat{G})$. We compute
\[
\begin{split}
4\pi^2 \EE^\bullet(\theta) &\leq \sum_{Q \subset U} 2 \int_Q \frac{1}{[(|w| \vee r) - 2\e]^2} \, dw \\
&\leq 2 \int_U \frac{1}{[(|w| \vee r) - 2\e]^2} \, dw \\
&= \frac{2\pi[r^2 - (r-2\e)^2]}{(r-2\e)^2} + 4\pi \int_r^{\diam(\widehat{G})} \frac{1}{(t - 2\e)^2} t\,dt \\
&\leq \frac{8\pi \e (r - \e)}{(r-2\e)^2} + 4\pi \int_{r - 2\e}^{\diam(\widehat{G})} \frac{s + 2\e}{s^2} \,ds \\
&\leq 16\pi + 4\pi \left[ \log\left( \frac{\diam(\widehat{G})}{r - 2\e} \right) + 2 \right]\, .
\end{split}
\]
Since $B \subset \Int(G)$, we have $\diam(\widehat{G}) \geq 2r$ and so $\log(\diam(\widehat{G})/r) \geq \log(2)$. This and the assumption $r \geq 3\e$ imply that $\log(\diam(\widehat{G})/(r-2\e))$ is bounded above by $C \log(\diam(\widehat{G})/r)$. Thus we conclude that
\[
\EE^\bullet(\theta) \leq C \log\left( \frac{\diam(\widehat{G})}{r} \right) \, . \qedhere
\]
\end{proof}

\begin{proof}[Proof of Proposition \ref{prop:left-to-right}]
We use the method of random paths \cite[p.~40]{LP16}. Let $\rho$ be a random variable supported on $[r_1,r_2]$ whose density at $t$ is proportional to $1/t$. Thus
\[
\Prob(a < \rho < b) = \frac{1}{Z} \int_a^b \1\{ t \in [r_1,r_2] \} \frac{dt}{t}\, ,
\]
with
\[
Z = \int_{r_1}^{r_2} \frac{dt}{t} = \log(r_2/r_1)\, .
\]
Because $\rho \in (r_1,r_2)$ almost surely, by assumption there is a simple path $P_\rho$ in $G^\bullet$ from $S$ to $T$ consisting entirely of $\rho$-edges. Since $\rho$ is a random variable, the path $P_\rho$ is random.

Define $\theta: \vec{E}^\bullet \to \R$ by
\[
\theta(e) = \Prob(e \in P_\rho) - \Prob(-e \in P_\rho)\, .
\]
We know that $\theta$ is a unit flow in $G^\bullet$ from $S$ to $T$, since it is a weighted average of the unit flows from $S$ to $T$ along the various options for the path $P_\rho$. As in the proof of Proposition \ref{prop:center-to-outside}, for each face $Q = [v_1,w_1,v_2,w_2]$ of $G$ we let $|\theta(e_Q^\bullet)|$ be the common value of $|\theta(e)|, |\theta(-e)|$ for both orientations $e,-e$ of $e_Q^\bullet$. Under the convention that $|w_1| \leq |w_2|$, the path $P_\rho$ can only pass through $e_Q^\bullet$ if $|w_1| < \rho \leq |w_2|$. Thus,
\[
\begin{split}
|\theta(e_Q^\bullet)| &\leq \Prob(\text{$P_\rho$ passes through $e_Q^\bullet$}) \\
&\leq \frac{1}{Z} \int_{|w_1|}^{|w_2|} \1\{ t \in [r_1,r_2] \} \frac{dt}{t} \\
&= \frac{1}{Z} (\tlog |w_2| - \tlog |w_1|)\, ,
\end{split}
\]
where we define the ``truncated log''
\[
\tlog(t) = \log((t \vee r_1) \wedge r_2)\, .
\]
The contribution of $e_Q^\bullet$ to $\EE^\bullet(\theta)$ is
\[
\frac{|v_1 v_2|}{|w_1 w_2|} |\theta(e_Q^\bullet)|^2 \leq \frac{1}{Z^2}(|v_1 v_2| \cdot |w_1 w_2|) \left( \frac{\tlog |w_2| - \tlog |w_1|}{|w_1 w_2|} \right)^2
\]
where again the quantity $|v_1 v_2| \cdot |w_1 w_2|$ is twice the area of $Q$.

The function $f(w) = \log |w|$ on $\R^2 \setminus \{0\}$ has $|\nabla f(w)| = 1/|w|$. Therefore, if $\rho_{\mathrm{min}} = \min\{ |w| : w \in w_1 w_2 \}$,
\[
\big| \tlog |w_2| - \tlog |w_1| \big| \leq \big| \log |w_2| - \log |w_1| \big| \leq |w_1 w_2| / \rho_\mathrm{min}\, .
\]
In addition, the Lipschitz constant of $w \mapsto \tlog |w|$ is $1/r_1$, so we also have
\[
\big| \tlog |w_2| - \tlog |w_1| \big| \leq |w_1 w_2| / r_1\, .
\]

Because each $w \in Q$ has $(|w| - 2\e)_+ \leq \rho_{\mathrm{min}}$ (where $t_+ = t \vee 0$),
\[
2\area(Q) \left( \frac{\tlog |w_2| - \tlog |w_1|}{|w_1 w_2|} \right)^2 \leq 2 \int_Q \left( \frac{1}{(|w| - 2\e)_+^2} \wedge \frac{1}{r_1^2} \right) dw\, .
\]
Let $D$ be the closed disk of radius $r_2 + 2\e$ centered at the origin. If $Q$ intersects $\R^2 \setminus D$, then $|w_1|,|w_2| > r_2$ and so $|\theta(e_Q^\bullet)| = 0$. Therefore,
\[
\begin{split}
Z^2 \EE^\bullet(\theta) &\leq \sum_{Q \subset D} 2 \int_Q \left( \frac{1}{(|w| - 2\e)_+^2} \wedge \frac{1}{r_1^2} \right) dw \\
&\leq 2 \int_D \left( \frac{1}{(|w| - 2\e)_+^2} \wedge \frac{1}{r_1^2} \right) dw \\
&= \frac{2}{r_1^2} \cdot \pi(r_1 + 2\e)^2 + 4\pi \int_{r_1 + 2\e}^{r_2 + 2\e} \frac{1}{(t - 2\e)^2} t\,dt \\
&\leq 18\pi + 4\pi \int_{r_1}^{r_2} \frac{s + 2\e}{s^2} \,ds \\
&\leq 4\pi \log(r_2/r_1) + 26\pi \\
&\leq C \log(r_2/r_1)\, .
\end{split}
\]
Since $Z = \log(r_2/r_1)$, the conclusion follows.
\end{proof}

\section{Equicontinuity of discrete harmonic functions}
\label{Equicontinuity section}

In this section, we prove the following statement. It is analogous to \cite[Lemma 2.4]{S13} and \cite[Proposition 4.3]{W15}, except that in those results, $G$ is required to satisfy certain regularity conditions and the constant $C$ depends on those conditions.

\begin{prop} \label{equicontinuity}
Let $G = (V^\bullet \sqcup V^\circ, E)$ be a finite orthodiagonal map with maximal edge length at most $\e$. Let $h: V^\bullet \to \R$ be discrete harmonic on $\Int(V^\bullet)$. Given $x,y \in V^\bullet$, set $r = \frac{1}{2}|x-y|$ and fix $R \geq 2r + 3\e$. Let $D$ be the closed disk of radius $R$ centered at $\frac{1}{2}(x+y)$, and set
\[
\beta = \max_{v,v' \in \partial V^\bullet \cap D} |h(v) - h(v')|\, ,
\]
with $\beta = 0$ if there are no such vertices. Then there exists a universal constant $C < \infty$ such that
\[
|h(x) - h(y)| \leq \frac{C \EE^\bullet(h)^{1/2}}{\log^{1/2}\left[ R \big/ (r + \e) \right]} + \beta\, .
\]
\end{prop}

To prove Proposition \ref{equicontinuity}, we may assume for convenience that the origin is located at $\frac{1}{2}(x+y)$. Then $|x| = |y| = r$ and the disk $D$ is centered at the origin. We also may assume without loss of generality that $h(x) < h(y)$.

The main step in the proof is an application of Proposition \ref{prop:left-to-right}, which is made possible by the following lemma. We refer to the beginning of Section \ref{sec:resistance} for the definition of a $\rho$-edge.

\begin{lemma}
\label{rho-path-application}
Under the assumptions of Proposition \ref{equicontinuity}, suppose that the origin is located at $\frac{1}{2}(x+y)$ and that $h(x) < h(y)$. Define
\begin{align*}
A &= \{ v \in V^\bullet : h(v) \leq h(x) \} \\
B &= \{ v \in V^\bullet : h(v) \geq h(y) \}
\end{align*}
and set $B' = B \cup (\partial V^\bullet \cap D)$. Assume that $A \cap B' = \varnothing$. Then, for every $\rho \in (r+\e, R-\e)$, there is a path in $G^\bullet$ from $A$ to $B'$ consisting entirely of $\rho$-edges.
\end{lemma}

We now show how Proposition \ref{equicontinuity} follows from Lemma \ref{rho-path-application} and Proposition \ref{prop:left-to-right}.

\begin{proof}[Proof of Proposition \ref{equicontinuity}]
We may assume that the origin is located at $\frac{1}{2}(x+y)$ and that $h(x) < h(y)$, so Lemma \ref{rho-path-application} applies. Define $A,B,B'$ as in the statement of that lemma. If $A \cap B' = \varnothing$, then we invoke Proposition \ref{prop:left-to-right} with $S = A$, $T = B'$, $r_1 = r+\e$, $r_2 = R-\e$ to find a unit flow $\theta$ in $G^\bullet$ from $A$ to $B'$ with
\begin{equation} \label{theta-low-energy}
\EE^\bullet(\theta) \leq \frac{C}{\log\left[ (R-\e) \big/ (r + \e) \right]} \leq \frac{C}{\log\left[ R \big/ (r + \e) \right]}
\end{equation}
(where the value of $C$ changes in the second inequality, using that $R \geq 2r + 3\e$). An exactly parallel argument (using a suitable modification of Lemma \ref{rho-path-application}) shows that if $A' \cap B = \varnothing$, where $A' = A \cup (\partial V^\bullet \cap D)$, then there is a unit flow $\varphi$ in $G^\bullet$ from $A'$ to $B$ with
\begin{equation} \label{phi-low-energy}
\EE^\bullet(\varphi) \leq \frac{C}{\log\left[ R \big/ (r + \e) \right]}\, .
\end{equation}

Assume first that $\partial V^\bullet \cap D$ is empty. Proposition \ref{Dirichlet Thomson} applied to $\theta$ and $h$ yields $h(y) - h(x) = \gap_{A,B}(h) \leq \EE^\bullet(\theta)^{1/2} \EE^\bullet(h)^{1/2}$. The upper bound \eqref{theta-low-energy} completes the proof.

Now suppose that $\partial V^\bullet \cap D$ is nonempty. If $\partial V^\bullet \cap D \subseteq B$, so that $B = B'$, then the argument in the last paragraph works without any changes. If $\partial V^\bullet \cap D \subseteq A$, then use the same argument with $\varphi$ and \eqref{phi-low-energy} instead of $\theta$ and \eqref{theta-low-energy}.

Finally, suppose that $\partial V^\bullet \cap D$ is not a subset of $A$ or of $B$. Let $\beta_1,\beta_2$ be the minimum and maximum, respectively, of $h(v)$ over $v \in \partial V^\bullet \cap D$. Then $\beta = \beta_2 - \beta_1$, and by assumption, $\beta_1 < h(y)$ and $\beta_2 > h(x)$. Let $\alpha_1 = \beta_1 - h(x)$ and $\alpha_2 = h(y) - \beta_2$. See Figure \ref{number line} for a visualization of the values of $h$. We seek an upper bound on $h(y) - h(x) - \beta = \alpha_1 + \alpha_2$.
\begin{figure}
\centering
\begin{tikzpicture}
\draw[<->] (0, 0) -- (6, 0);
\draw (1, -0.1) -- (1, 0.1);
\draw (2, -0.1) -- (2, 0.1);
\draw (3.8, -0.1) -- (3.8, 0.1);
\draw (5, -0.1) -- (5, 0.1);
\draw[very thick] (2, 0) -- (3.8, 0);

\node[below] at (1, -0.1) {\small $h(x)$};
\node[below] at (2, -0.1) {\small $\beta_1$};
\node[below] at (3.8, -0.1) {\small $\beta_2$};
\node[below] at (5, -0.1) {\small $h(y)$};
\node[above] at (1.5, 0) {\small $\alpha_1$};
\node[above] at (2.9, 0) {\small $\beta$};
\node[above] at (4.4, 0) {\small $\alpha_2$};
\end{tikzpicture}
\caption{Number line showing the values of $h$ in the case where $\alpha_1, \alpha_2 > 0$. The thick segment is the range of values of $h(v)$ for $v \in \partial V^\bullet \cap D$.}
\label{number line}
\end{figure}
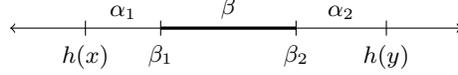

If $\alpha_1 \leq 0$, we bound it by $0$. If $\alpha_1 > 0$, then $A \cap B' = \varnothing$ and Proposition \ref{Dirichlet Thomson} gives $\alpha_1 = \gap_{A,B'}(h) \leq \EE^\bullet(\theta)^{1/2} \EE^\bullet(h)^{1/2}$. Similarly, if $\alpha_2 \leq 0$, we bound it by $0$, while if $\alpha_2 > 0$, then $A' \cap B = \varnothing$ and Proposition \ref{Dirichlet Thomson} gives $\alpha_2 = \gap_{A',B}(h) \leq \EE^\bullet(\varphi)^{1/2} \EE^\bullet(h)^{1/2}$. To conclude, use the bounds \eqref{theta-low-energy} and \eqref{phi-low-energy}.
\end{proof}

It remains to prove Lemma \ref{rho-path-application}. The short version of the proof is that by the discrete maximum principle, both sets $A$ and $B$ must extend to $\partial V^\bullet$. Once we know this, the conclusion follows from general properties of plane graphs and their duals and has nothing to do with orthodiagonal maps specifically. To emphasize this point, we now write the required statement about plane graphs as a freestanding lemma. We have defined $\rho$-edges for orthodiagonal maps; in general, if $H$ is a finite connected plane graph with plane dual $H^\dagger$, a $\rho$-edge of $H$ is an edge whose corresponding dual edge connects two vertices $w,w'$ of $H^\dagger$ with $|w| < \rho \leq |w'|$.

\begin{lemma}
\label{rho-path}
Let $H$ be a finite connected plane graph with plane dual $H^\dagger$. Let $W,W^\dagger$ be the vertex sets of $H,H^\dagger$. For each $w \in W^\dagger$, let $f(w)$ denote the face of $H$ that contains $w$. Assume that each inner face $f(w)$ of $H$ is contained in the open disk of radius $\e$ centered at $w$. Fix $0 < r < R$, with $r+\e < R-\e$, and let $D$ be the closed disk of radius $R$ centered at the origin. Let $z \in W^\dagger$ be the dual vertex contained in the outer face of $H$, and assume that $z \notin D$.

Denote by $\partial W$ the vertices in $W$ that are on the boundary of the outer face of $H$. Let $x,y \in W$ satisfy $|x|,|y| \leq r$. Let $A,B$ be subsets of $W$ such that there is a path in $H$ from $x$ to $\partial W$ whose vertices are all in $A$, and there is a path in $H$ from $y$ to $\partial W$ whose vertices are all in $B$.

Set $B' = B \cup (\partial W \cap D)$. If $A \cap B' = \varnothing$, then for each $\rho \in (r+\e, R-\e)$, there is a path in $H$ from $A$ to $B'$ consisting entirely of $\rho$-edges.
\end{lemma}

\begin{proof}
Since $x \in A \cap D$, which is disjoint from $\partial W$, we have $x \notin \partial W$. Let $u \in W^\dagger$ be a dual vertex such that $f(u)$ is incident to $x$. Then $u \neq z$ and $|u| \leq r + \e < \rho$. Let $S_\rho$ be the set of vertices $w \in W^\dagger$ such that there is a path from $u$ to $w$ in $H^\dagger$, all of whose vertices including $w$ are in the open disk of radius $\rho$ centered at the origin. Note that $z \notin S_\rho$ since $|z| > R > \rho$. Although we will not need this fact in the rest of the proof, it is the case that $S_\rho$ does not depend on the particular choice of $u$. Indeed, the dual vertices $u'$ with $f(u')$ incident to $x$ form a (not necessarily simple) cycle that bounds the face of $H^\dagger$ containing $x$, and drawing paths from $u$ along this cycle shows that all of these vertices $u'$ are in $S_\rho$.

Let $H[S_\rho]$ be the subgraph of $H$ whose vertices and edges are those that make up the boundaries of the faces $f(w)$, for $w \in S_\rho$. See Figure \ref{HS-rho}. The graph $H[S_\rho]$ is connected. In addition, for each $w \in S_\rho$ we have $|w| < \rho < R-\e$, implying that $f(w)$ is contained inside $D$. Hence all the vertices and edges of $H[S_\rho]$ are in $D$. Let $\fout_\rho$ be the outer face of $H[S_\rho]$, which contains the complement of $D$. Because $H[S_\rho]$ is connected, the boundary of $\fout_\rho$ is a connected subgraph which we label $K$. Every edge $e$ of $K$ borders two faces of $H[S_\rho]$: an inner face $f(w)$ and $\fout_\rho$. Therefore, the corresponding dual edge $e^\dagger$ of $H^\dagger$ has endpoints $w,w' \in W^\dagger$, where $w \in S_\rho$ and $w' \notin S_\rho$. We have $|w| < \rho$. Since $w'$ is adjacent to $w$ in $H^\dagger$, if $|w'| < \rho$ then $w' \in S_\rho$, which is not true. Thus $|w'| \geq \rho$, and $e$ is a $\rho$-edge of $H$.

\begin{figure}
\centering
\includegraphics[width=0.8\textwidth]{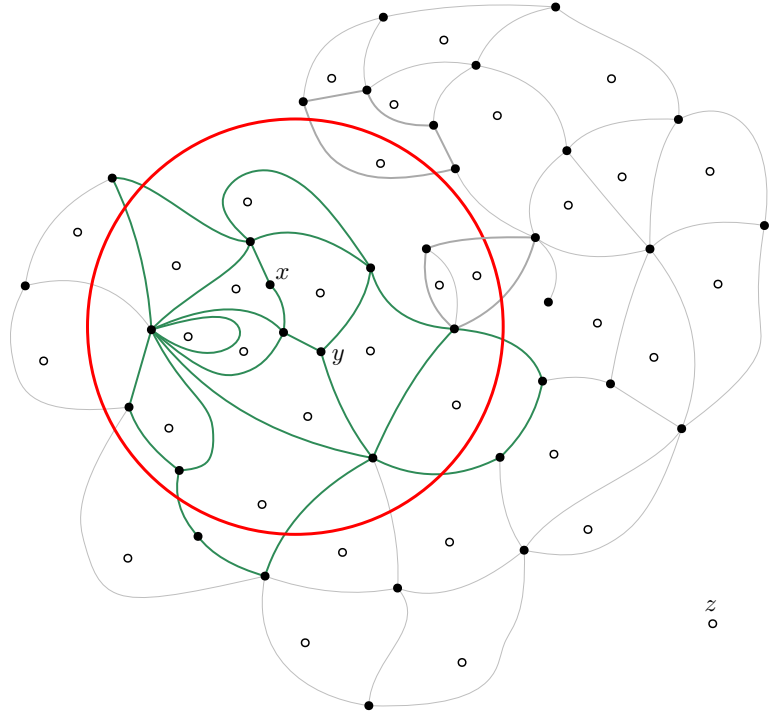}
\caption{Construction of the graph $H[S_\rho]$. The circle of radius $\rho$ centered at the origin is drawn in red. Solid disks represent vertices of $H$, while hollow disks represent vertices of $H^\dagger$. The edges of $H[S_\rho]$ are drawn in green, and the other edges of $H$ are drawn in gray. For clarity, the edges of $H^\dagger$ are not shown. The $\rho$-edges of $H$ are the edges in the boundary of the outer face of $H[S_\rho]$ along with the seven edges drawn in darker gray.}
\label{HS-rho}	
\end{figure}

We know that $K$ is a connected subgraph of $H$ and all of its edges are $\rho$-edges. If we show that $K$ contains a vertex in $A$ and a vertex in $B'$, then the path in $K$ between these vertices will fulfill the requirements. Let $x',y' \in \partial W$ be respectively the endpoint of the path from $x$ to $\partial W$ whose vertices are in $A$, and the endpoint of the path from $y$ to $\partial W$ whose vertices are in $B$. Since $A$ is disjoint from $\partial W \cap D$, we have $x' \notin D$, so $x' \in \fout_\rho$. Meanwhile, $x \notin \fout_\rho$ because $x$ is a vertex of $H[S_\rho]$. The path from $x$ to $x'$ must therefore intersect the boundary $K$ of $\fout_\rho$. Because both the path and $K$ are subgraphs of $H$, this intersection contains a vertex, which is in $A$.

To show that $K$ contains a vertex in $B'$, first suppose that $y \notin \fout_\rho$. If $y' \in \fout_\rho$, then the path from $y$ to $y'$ intersects $K$ at a vertex in $B$, as above. If $y' \notin \fout_\rho$, then we observe that the outer face of $H$ is incident to $y'$ and is a subset of $\fout_\rho$. Hence $y'$ itself must be a vertex of $K$ that is in $B$.

Suppose now that $y \in \fout_\rho$. As $x \notin \fout_\rho$, the straight line $xy$ intersects $K$ at some point $p$ with $|p| \leq r$. Let $e$ be the edge of $K$ that contains $p$, or an edge of $K$ incident to $p$ in case $p$ is at a vertex of $K$. We have seen that $e$ is a $\rho$-edge of $H$, meaning that the endpoints of its corresponding dual edge satisfy $|w| < \rho \leq |w'|$. If $w' \neq z$, so that $f(w')$ is an inner face of $H$, then the distance between $w'$ and $p$ is at most $\e$. This contradicts that $|p| \leq r$ and $|w'| \geq \rho > r+\e$. We conclude that $w' = z$ and $e$ is part of the boundary of the outer face of $H$. The two endpoints of $e$ are vertices of $K$ that are in $\partial W$, and they are also in $D$ because all the vertices of $K$ are contained in $D$. Therefore, they are elements of $\partial W \cap D \subseteq B'$.
\end{proof}

\begin{proof}[Proof of Lemma \ref{rho-path-application}]
In order to apply Lemma \ref{rho-path}, 	we define augmented versions of the graphs $G^\bullet$ and $G^\circ$ that are exact plane duals. Let
\[
((v_1,w_1), (w_1,v_2), (v_2,w_2), (w_2,v_3), \ldots, (v_n,w_n), (w_n,v_1))
\]
be the oriented boundary of the outer face $f$ of $G$ (in either direction), with each $v_j \in \partial V^\bullet$ and each $w_j \in \partial V^\circ$. For each $j$, we draw a new primal edge $e^\bullet_j$ between $v_j$ and $v_{j+1}$ such that $e^\bullet_j$ is contained in $f$ and separates $w_j$ from infinity. (Indices are taken mod $n$.) We also draw a new dual vertex $z$ located in the new outer face, with $|z| > R$. For each $j$, we draw a dual edge $e^\circ_j$ between $w_j$ and $z$ that intersects $e^\bullet_j$ at a single point. See Figure \ref{orthomap-aug}. The \textbf{augmented primal graph} $\bar{G}^\bullet = (\bar{V}^\bullet, \bar{E}^\bullet)$ is defined by $\bar{V}^\bullet = V^\bullet$ and $\bar{E}^\bullet = E^\bullet \cup \bigcup_j e^\bullet_j$. The \textbf{augmented dual graph} $\bar{G}^\circ = (\bar{V}^\circ, \bar{E}^\circ)$ is defined by $\bar{V}^\circ = V^\circ \cup \{z\}$ and $\bar{E}^\circ = E^\circ \cup \bigcup_j e^\circ_j$. With these definitions, $\bar{G}^\bullet$ and $\bar{G}^\circ$ are plane duals of each other. We observe that the vertices in the boundary of the outer face of $\bar{G}^\bullet$ are precisely the elements of $\partial V^\bullet$; this was not the case for the un-augmented primal graph $G^\bullet$.

\begin{figure}
\centering
\begin{subfigure}{1\textwidth}
   \centering
   \includegraphics[width=0.7\textwidth]{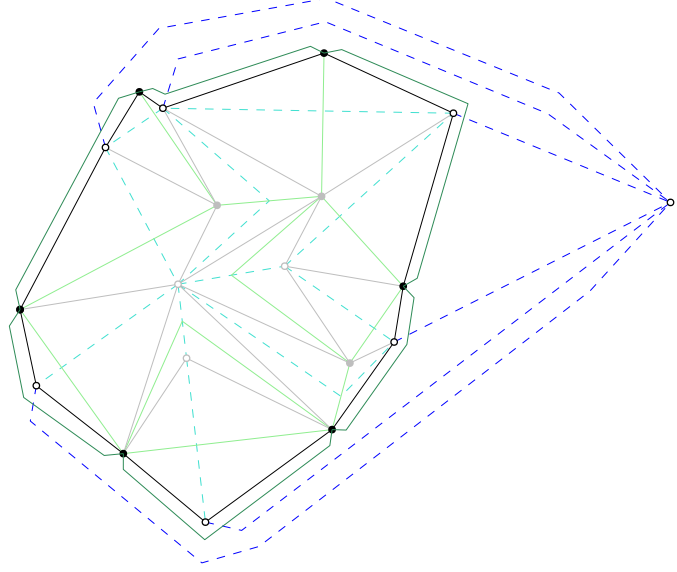}
   \caption{Augmented primal and dual graphs for the orthodiagonal map from Figure \ref{orthomap}. The boundary $\partial G$ is in black. The new primal edges in $\bar{E}^\bullet \setminus E^\bullet$ are in solid green, and the new dual edges in $\bar{E}^\circ \setminus E^\circ$ are in dashed blue. The interior of $G$ is shown with edges of $G$, $G^\bullet$, and $G^\circ$ in light gray, light green, and dashed light blue, respectively.}
   \label{orthomap-aug}
\end{subfigure}

\vspace{\floatsep}

\begin{subfigure}{1\textwidth}
   \centering
   \includegraphics[width=0.2\textwidth]{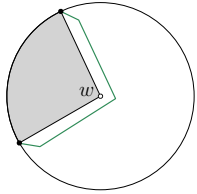}
   \caption{The circle of radius $\e$ centered at a vertex $w \in \partial V^\circ$, with its two primal neighbors along the boundary $\partial G$. The shaded region is $\widehat{G}$. If the new primal edge (in green) is drawn completely inside the circle, as shown, then the face $f^\bullet(w)$ of $\bar{G}^\bullet$ will be contained in the disk of radius $\e$ centered at $w$. Note that the two primal vertices need not be on the circle but might be inside.}
   \label{ortho-aug2}
\end{subfigure}

\caption{The augmented primal and dual graphs.}
\label{augmented}	
\end{figure}

Given $w \in \bar{V}^\circ$, let $f^\bullet(w)$ denote the face of $\bar{G}^\bullet$ that contains $w$. For each $w \in \Int(V^\circ)$, the face $f^\bullet(w)$ is contained in the open disk of radius $\e$ centered at $w$. We would like the same property to hold for each $w_j \in \partial V^\circ$. This can easily be ensured by drawing the new primal edges $e_j^\bullet$ appropriately: see Figure \ref{ortho-aug2}.

We will apply Lemma \ref{rho-path} with $H = \bar{G}^\bullet$ and $H^\dagger = \bar{G}^\circ$. (The letters $R,r,\e,x,y$ keep their meanings which were inherited from the statement of Proposition \ref{equicontinuity}.) Under this choice of $H$, the set $\partial W$ in Lemma \ref{rho-path} is our $\partial V^\bullet$, so the two definitions of $B'$ match. We must check that $\bar{G}^\bullet$ contains a path from $x$ to $\partial V^\bullet$ whose vertices are all in $A$ and a path from $y$ to $\partial V^\bullet$ whose vertices are all in $B$.

Let $A_1$ be the connected component of $A$ in $G^\bullet$ that contains $x$. Assume for contradiction that $A_1 \subseteq \Int(V^\bullet)$, so that $h$ is discrete harmonic on $A_1$. Because $G^\bullet$ is connected, the set $A_2$ of vertices $v \in V^\bullet$ such that $v \notin A_1$ and $v$ is adjacent in $G^\bullet$ to some vertex of $A_1$ is nonempty. The discrete maximum principle, Proposition \ref{max principle}, implies that there is $v \in A_2$ such that $h(v) \leq h(x)$. Then $v \in A$ and $v$ is adjacent to a vertex of $A_1$, so in fact $v \in A_1$, a contradiction. We conclude that $A_1$ contains a vertex in $\partial V^\bullet$, meaning that there is a path in $G^\bullet$ from $x$ to $\partial V^\bullet$ whose vertices are all in $A$. By the same argument, there is also a path in $G^\bullet$ from $y$ to $\partial V^\bullet$ whose vertices are all in $B$.

All the conditions of Lemma \ref{rho-path} have now been met. It follows that for every $\rho \in (r+\e, R-\e)$, there is a path $P_\rho$ in $\bar{G}^\bullet$ from $A$ to $B'$ consisting entirely of $\rho$-edges. We may assume without loss of generality that $P_\rho$ is a simple path whose internal vertices are in neither $A$ nor $B'$. This means that $P_\rho$ cannot include any of the extra edges in $\bar{E}^\bullet \setminus E^\bullet$: if $e$ is a $\rho$-edge in $\bar{E}^\bullet \setminus E^\bullet$, then both of its endpoints are in $\partial V^\bullet \cap D \subseteq B'$, so $e$ cannot be an edge of $P_\rho$. Hence $P_\rho$ is in fact a path in $G^\bullet$. Note that for every edge in $E^\bullet$, its dual edge in $\bar{E}^\circ$ according to the duality between $\bar{G}^\bullet$ and $\bar{G}^\circ$ is the same as its dual edge in $E^\circ$ according to the pre-existing correspondence between $E^\bullet$ and $E^\circ$. Therefore, the two definitions of $\rho$-edge (for the primal graph of an orthodiagonal map and for a general connected plane graph) are equivalent in this setting.
\end{proof}

\section{Proof of Theorem \ref{main_thm}}
\label{main proof}

With the results from Sections \ref{Energy convergence} through \ref{Equicontinuity section} in hand, we are ready to prove Theorem \ref{main_thm}. The assumptions of the theorem are as follows:
\begin{enumerate}[label=(A\arabic*)]
\item The domain $\Omega \subset \R^2$ is bounded and simply connected.
\item The orthodiagonal map $G = (V^\bullet \sqcup V^\circ, E)$ has maximal edge length at most $\e$, and the Hausdorff distance between $\partial G$ and $\partial \Omega$ is at most $\delta$, where both $\e$ and $\delta$ are less than $\diam(\Omega)$.
\item The function $g: \R^2 \to \R$ is $C^2$, with
\[
C_1 = \|\nabla g\|_{\infty, \widetilde{\Omega}}\, , \qquad C_2 = \|Hg\|_{\infty, \widetilde{\Omega}}
\]
where $\widetilde{\Omega} = \conv(\overline{\Omega} \cup \widehat{G})$.
\end{enumerate}
Under these assumptions, one defines functions $h_c,h_d$ such that:
\begin{enumerate}[label=(A\arabic*)]
\setcounter{enumi}{3}
\item The function $h_c: \overline{\Omega} \to \R$ is the solution to the continuous Dirichlet problem on $\Omega$ with boundary data $g$.
\item The function $h_d: V^\bullet \to \R$ is the solution to the discrete Dirichlet problem on $\Int(V^\bullet)$ with boundary data $g|_{\partial V^\bullet}$.
\end{enumerate}
Given (A1)-(A5), Theorem \ref{main_thm} says that for all $x \in V^\bullet \cap \overline{\Omega}$,
\begin{equation} \label{main conclusion}
|h_d(x) - h_c(x)| \leq \frac{C \diam(\Omega) (C_1 + C_2 \e)}{\log^{1/2}(\diam(\Omega) / (\delta \vee \e))}\, .
\end{equation}

The inequality \eqref{main conclusion} is invariant under spatial scaling. If space is scaled by a factor of $\alpha$, then $\diam(\Omega), \delta, \e$ are multiplied by $\alpha$ while $C_1$ is multiplied by $\alpha^{-1}$ and $C_2$ is multiplied by $\alpha^{-2}$. Thus neither the left side nor the right side of \eqref{main conclusion} changes. For this reason, we may assume in the proof that:
\begin{enumerate}[label=(A\arabic*)]
\setcounter{enumi}{5}
\item The diameter of $\Omega$ is $1$.	
\end{enumerate}
In that case, we have $\delta,\e < 1$ and we are proving that
\begin{equation} \label{main conclusion 2}
|h_d(x) - h_c(x)| \leq \frac{C(C_1 + C_2 \e)}{\log^{1/2}(1 / (\delta \vee \e))}
\end{equation}
for all $x \in V^\bullet \cap \overline{\Omega}$.

To prove the theorem, we will convert the $L^2$ bound of Proposition \ref{energy convergence} into an $L^\infty$ bound. Proposition \ref{energy convergence} does not apply directly to the functions $h_c$ and $h_d$ from (A4) and (A5), since the set $\widehat{G}$ may extend outside $\overline{\Omega}$. In addition, we have no control over $\|Hh_c\|_{\infty, \Omega}$, which might in fact be infinite. We will resolve both of these issues by defining a sub-orthodiagonal map $G_1$ of $G$ such that $\widehat{G}_1$ is contained inside $\Omega$ with a buffer separating $\widehat{G}_1$ from $\partial \Omega$. The norm of $Hh_c$ on $\widehat{G}_1$ (and on the slightly larger set $\widetilde{G}_1$ appearing in Proposition \ref{energy convergence}) is then controlled using the interior derivative estimates, Proposition \ref{interior derivative estimates}.

To be precise, given a finite orthodiagonal map $G = (V^\bullet \sqcup V^\circ, E)$, we say that $G_1 = (V_1^\bullet \sqcup V_1^\circ, E_1)$ is a \textbf{sub-orthodiagonal map} of $G$ if $G_1$ is itself an orthodiagonal map and we have the inclusions $V_1^\bullet \subseteq V^\bullet$, $V_1^\circ \subseteq V^\circ$, $E_1 \subseteq E$. In addition, we require that every inner face of $G_1$ is also a face of $G$.

Once we are in position to apply Proposition \ref{energy convergence}, we could show that $h_d$ is uniformly close to $h_c$ using Proposition \ref{equicontinuity} and the Arzel\`{a}-Ascoli theorem, but this would not give an effective bound. The proof of Theorem \ref{main_thm} relies on Propositions \ref{energy convergence} and \ref{equicontinuity} but also uses the resistance bound of Proposition \ref{prop:center-to-outside}. In addition to these main ingredients, there are other estimates which we state as claims and whose proofs we postpone until the end of the subsection.

\begin{proof}[Proof of Theorem \ref{main_thm}]
We are given (A1)-(A5), and by scale-invariance we may also assume (A6). The goal is to prove \eqref{main conclusion 2} for all $x \in V^\bullet \cap \overline{\Omega}$. Set $\gamma = \delta \vee \e$. By (A2) and (A6), we have $\gamma < 1$.

We begin with the observation that since $h_c$ and $h_d$ satisfy approximately the same boundary conditions and the smoothness of both functions is controlled (by Propositions \ref{Holder} and \ref{equicontinuity}, respectively), the difference $|h_c - h_d|$ must be relatively small near $\partial \Omega$. Specifically, we can say the following.

\begin{claim} \label{claim:boundary}
The bound \eqref{main conclusion 2} holds for all $x \in V^\bullet \cap \overline{\Omega}$ such that
\[
\dist(x, \partial \Omega) \leq 2 \gamma^{1/3} \, .
\]
\end{claim}

Assume for the rest of the proof that the set $\{x \in V^\bullet \cap \overline{\Omega} : \dist(x, \partial \Omega) > 2 \gamma^{1/3} \}$ is nonempty. Let $S$ be the set of all inner faces $Q$ of $G$ such that $Q \subset \Omega$ and $\dist(Q, \partial \Omega) \geq \gamma^{1/3}$, and let $G[S]$ be the subgraph of $G$ that is the union of the boundaries of the elements of $S$. The construction of $G[S]$ is shown in Figure \ref{G1-fig}.

\begin{figure}
\centering
\includegraphics[width=0.8\textwidth]{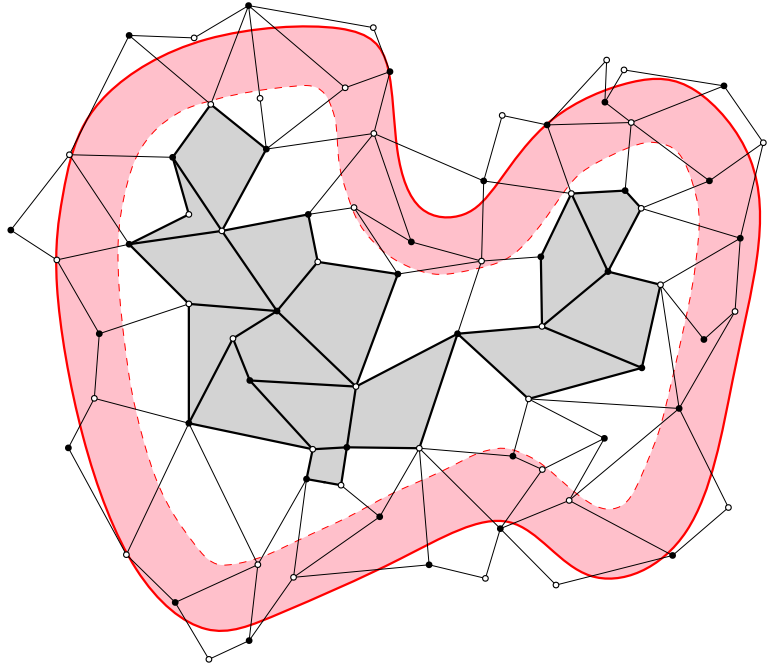}
\caption{Construction of the graph $G[S]$. The solid red curve is $\partial \Omega$ and the red shaded region is the set of points in $\Omega$ whose distance from $\partial \Omega$ is less than $\gamma^{1/3}$. The gray shaded faces are the elements of $S$, and the thick black lines are the edges of $G[S]$. Both blocks of $G[S]$ are sub-orthodiagonal maps of $G$. Warning: For this figure to be completely accurate, the buffer width $\gamma^{1/3}$ would need to be much larger than the maximum edge length $\e$.}
\label{G1-fig}	
\end{figure}

\begin{claim} \label{claim:G[S]}
The set $S$ is nonempty. The inner faces of $G[S]$ are precisely the elements of $S$, and every point $p$ on the boundary of the outer face of $G[S]$ satisfies
\begin{equation} \label{dG1}
\dist(p, \partial \Omega) \leq (9/8) \gamma^{1/3}\, .
\end{equation}

Each $x \in V^\bullet \cap \overline{\Omega}$ with $\dist(x, \partial \Omega) > 2\gamma^{1/3}$ is a vertex of a unique block of $G[S]$. If $H$ is a block of $G[S]$, then the inner faces of $H$ are all elements of $S$ and every point $p$ on the boundary of the outer face of $H$ satisfies \eqref{dG1}.
\end{claim}

A consequence of Claim \ref{claim:G[S]} is that each block $H$ of $G[S]$ is a sub-orthodiagonal map of $G$. Indeed, since $H$ is $2$-connected, the boundary of its outer face is a simple cycle \cite[Proposition 4.2.5]{D17} and the other required properties are all stated in the claim. Hence, even though $G[S]$ may violate the ``simple boundary'' condition of orthodiagonal maps, we can apply the results from Sections \ref{Energy convergence} through \ref{Equicontinuity section} to each block separately.

From this point forward we fix a particular $x \in V^\bullet \cap \overline{\Omega}$ with $\dist(x, \partial \Omega) > 2 \gamma^{1/3}$ and let $G_1$ be the block of $G[S]$ containing $x$. We write its vertex set as $V_1^\bullet \sqcup V_1^\circ$ (inheriting the bipartition from $G$) and its edge set as $E_1$, so that $G_1 = (V_1^\bullet \sqcup V_1^\circ, E_1)$ is a sub-orthodiagonal map of $G$ with $x \in V_1^\bullet$. We define the notations $\widehat{G}_1$, $\partial G_1$, $\partial V_1^\bullet$, $\Int(V_1^\bullet)$, etc.\ for $G_1$ analogously as for $G$. The primal network associated with $G_1$ is $(G_1^\bullet, c)$, with energy functional $\EE_1^\bullet$. We observe that for every vertex in $\Int(V_1^\bullet)$, its immediate neighborhood in $G_1^\bullet$ is identical to its immediate neighborhood in $G^\bullet$, and the edge weights are equal in both networks (justifying the reuse of the letter $c$).

Since $\widehat{G}_1$ is the union of closures of faces in $S$, we have that $\widehat{G}_1 \subset \Omega$ and $\dist(\widehat{G}_1, \partial \Omega) \geq \gamma^{1/3}$. The buffer between $\widehat{G}_1$ and $\partial \Omega$ lets us use the interior derivative estimates, Proposition \ref{interior derivative estimates}, to deduce the following bounds.

\begin{claim} \label{claim:deriv-Hessian}
Let $\widetilde{G}_1$ denote the union of the convex hulls of the closures of the inner faces of $G_1$. Then
\begin{align}
\|\nabla h_c\|_{\infty,\widehat{G}_1} &\leq CC_1 \gamma^{-1/3}\, , \label{deriv-bound} \\
\|Hh_c\|_{\infty,\widetilde{G}_1} &\leq CC_1 \gamma^{-2/3}\, . \label{Hessian-bound}
\end{align}
\end{claim}

Let $h_d^{(1)}: V_1^\bullet \to \R$ be the solution to the discrete Dirichlet problem on $\Int(V_1^\bullet)$ with boundary data $h_c|_{\partial V_1^\bullet}$. We write
\begin{equation} \label{triangle}
|h_d(x) - h_c(x)| \leq |h_d(x) - h_d^{(1)}(x)| + |h_d^{(1)}(x) - h_c(x)|\, .
\end{equation}
The first term on the right hand side is easily bounded using the discrete maximum principle and Claim \ref{claim:boundary}.

\begin{claim} \label{claim:hd-hd1}
We have
\[
|h_d(x) - h_d^{(1)}(x)| \leq \frac{C(C_1 + C_2 \e)}{\log^{1/2}(1/\gamma)}\, .
\]
\end{claim}

To deal with the second term on the right side of \eqref{triangle}, we let $f = h_d^{(1)} - h_c$. Thus we would like to bound $|f(x)|$ and we know that $f = 0$ on $\partial V_1^\bullet$. Let $B$ be the closed disk of radius $3\e$ centered at $x$, and let $A = B \cap V_1^\bullet$.

\begin{claim} \label{claim:A-smooth}
The set $B$ is a subset of $\Int(G_1)$. As well, every $y \in A$ satisfies
\begin{equation} \label{A-smoothness-formula}
|f(x) - f(y)| \leq \frac{CC_1}{\log^{1/2}(1/\e)}\, .
\end{equation}
\end{claim}

We now invoke Propositions \ref{energy convergence} and \ref{prop:center-to-outside}. Proposition \ref{energy convergence} and \eqref{Hessian-bound} imply that
\[
\EE_1^\bullet(f) \leq 32 \area(\widehat{G}_1) \|Hh_c\|_{\infty, \widetilde{G}_1}^2 \e^2 \leq CC_1^2 \e^{2/3}\, .
\]
Proposition \ref{prop:center-to-outside} shows that there is a unit flow $\theta$ in $G_1^\bullet$ from $A$ to $\partial V_1^\bullet$ such that
\[
\EE_1^\bullet(\theta) \leq C \log(1/\e)\, .
\]
We apply Proposition \ref{Dirichlet Thomson} using $f$ and $-\theta$, which is a unit flow from $\partial V_1^\bullet$ to $A$. This yields
\[
1 \cdot \gap_{\partial V_1^\bullet, A}(f) \leq \EE_1^\bullet(-\theta)^{1/2} \EE_1^\bullet(f)^{1/2} \leq CC_1 \e^{1/3} \log^{1/2}(1/\e)\, .
\]
(Proposition \ref{Dirichlet Thomson} only applies when the gap is nonnegative, but if the gap is negative then the inequality is trivially true.) Since $f = 0$ on $\partial V_1^\bullet$, there is $y \in A$ such that
\[
f(y) \leq CC_1 \e^{1/3} \log^{1/2}(1/\e)\, .
\]
Write $f(x) = f(y) + [f(x) - f(y)]$. By the above and \eqref{A-smoothness-formula},
\[
f(x) \leq CC_1 \e^{1/3} \log^{1/2}(1/\e) + \frac{CC_1}{\log^{1/2}(1/\e)} \leq \frac{CC_1}{\log^{1/2}(1/\e)}\, .
\]
We can make the same argument with $-f$ in place of $f$. Combining the two bounds,
\begin{equation} \label{eq:f-bound}
|f(x)| \leq \frac{CC_1}{\log^{1/2}(1/\e)}\, .
\end{equation}
Looking at the right side of \eqref{triangle}, we bound the first term using Claim \ref{claim:hd-hd1} and the second term using \eqref{eq:f-bound}. This completes the proof.
\end{proof}

It remains to prove Claims \ref{claim:boundary} through \ref{claim:A-smooth}.

\begin{proof}[Proof of Claim \ref{claim:boundary}]
Choose $u \in \partial\Omega$ such that $|x-u| \leq 2\gamma^{1/3}$. There is $v \in \partial G$ such that $|u-v| \leq \delta$, and there is $w \in \partial V^\bullet$ such that $|v-w| \leq \e$. See Figure \ref{xuvw}.

\begin{figure}
\centering
\includegraphics[width=0.4\textwidth]{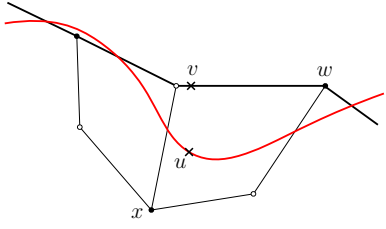}
\caption{Placement of the points $u,v,w$ relative to $x$. The red curve is $\partial \Omega$ and the thick black segments are part of $\partial G$.}
\label{xuvw}
\end{figure}

We have $|x-w| \leq 2\gamma^{1/3} + 2\gamma$. Since $h_d(w) = g(w)$ and $h_c(u) = g(u)$,
\begin{equation} \label{triangle-2}
|h_d(x) - h_c(x)| \leq |h_d(x) - h_d(w)| + |g(w) - g(u)| + |h_c(u) - h_c(x)|\, .
\end{equation}
The function $g$ has Lipschitz constant at most $C_1$ on the convex set $\widetilde{\Omega}$. Therefore,
\begin{equation} \label{g-bound}
|g(w) - g(u)| \leq C_1 |w-u| \leq 2C_1 \gamma\, .
\end{equation}
We can also apply Proposition \ref{Holder} to obtain
\begin{equation} \label{h_c-bound}
|h_c(u) - h_c(x)| \leq CC_1 |u-x|^{1/2} \leq CC_1 \gamma^{1/6}\, .
\end{equation}
To bound $|h_d(x) - h_d(w)|$, we use Proposition \ref{equicontinuity} with $r = \frac{1}{2}|x-w|$ and $R = 7\gamma^{1/6}$. Using that $\gamma < 1$, we have
\[
2r + 3\e \leq 2\gamma^{1/3} + 5\gamma \leq 7\gamma^{1/6} = R
\]
and
\[
\frac{R}{r+\e} \geq \frac{7\gamma^{1/6}}{\gamma^{1/3} + 2\gamma} \geq \frac{7\gamma^{1/6}}{3\gamma^{1/3}} \geq \gamma^{-1/6}\, .
\]
Because $h_d = g$ on $\partial V^\bullet$, the term $\beta$ from Proposition \ref{equicontinuity} satisfies
\[
\beta \leq 2RC_1 = 14C_1 \gamma^{1/6}\, .
\]
As well, Propositions \ref{projection properties} and \ref{sum integral convergence} imply that
\[
\begin{split}
\EE^\bullet(h_d) \leq \EE^\bullet(g) &\leq 2 \int_{\widehat{G}} |\nabla g|^2 + 2\area(\widehat{G})(10C_1 C_2 \e + 8C_2^2 \e^2) \\
&\leq 2\area(\widehat{G})(C_1^2 + 10C_1 C_2 \e + 8C_2^2 \e^2) \\
&\leq C (C_1 + C_2\e)^2\, .
\end{split}
\]
In the last inequality, we used that $\area(\widehat{G}) \leq C$, which is true since $\diam(\Omega) = 1$ and $\partial G$ is within Hausdorff distance $\delta < 1$ of $\partial \Omega$. Proposition \ref{equicontinuity} now gives
\begin{equation} \label{h_d-bound}
|h_d(x) - h_d(w)| \leq \frac{C(C_1 + C_2 \e)}{\log^{1/2}(1/\gamma)} + 14C_1 \gamma^{1/6}\, .
\end{equation}
Comparing the sizes of the bounds in \eqref{g-bound}, \eqref{h_c-bound}, and \eqref{h_d-bound}, we have
\[
\gamma \leq \gamma^{1/6} \leq \frac{C}{\log^{1/2}(1/\gamma)}\, .
\]
Therefore, we may sum the bounds to conclude by \eqref{triangle-2} that
\[
|h_d(x) - h_c(x)| \leq \frac{C(C_1 + C_2 \e)}{\log^{1/2}(1/\gamma)}\, . \qedhere
\]
\end{proof}

In the remaining claims, we know that there is $x \in V^\bullet \cap \overline{\Omega}$ satisfying
\[
\dist(x, \partial \Omega) > 2 \gamma^{1/3}\, .
\]
We have $x \in \Int(V^\bullet)$, since every vertex in $\partial V^\bullet$ is within distance $\delta \leq \gamma^{1/3}$ of $\partial \Omega$. As well, the disk of radius $2\gamma^{1/3}$ centered at $x$ is entirely within $\Omega$, so $1 = \diam(\Omega) \geq 4\gamma^{1/3}$ and thus
\begin{equation} \label{64-bound}
\gamma \leq \frac{1}{64}\, .
\end{equation}
We will often use the equivalent statement that
\begin{equation} \label{16-bound}
\gamma \leq \frac{1}{16}\gamma^{1/3}\, .
\end{equation}

\begin{proof}[Proof of Claim \ref{claim:G[S]}]
Suppose that $x \in V^\bullet \cap \overline{\Omega}$ with $\dist(x, \partial \Omega) > 2 \gamma^{1/3}$. Every face $Q$ of $G$ that is incident to $x$ satisfies
\[
\dist(Q, \partial \Omega) > 2\gamma^{1/3} - 2\e \geq \frac{15}{8} \gamma^{1/3}\, ,
\]
using \eqref{16-bound} in the last inequality, so $Q$ is an element of $S$. The boundaries of all the faces $Q$ incident to $x$ are contained in the same block of $G[S]$, which is the unique block that has $x$ as a vertex.

We now show that the inner faces of $G[S]$ are precisely the elements of $S$. Each element of $S$ is a face of $G[S]$. Conversely, let $f$ be a face of $G[S]$. We will prove that either $f \in S$ or $f = \fout$, the outer face of $G[S]$. We know that $f$ contains at least one face of $G$. If $f$ contains the outer face of $G$, then $f = \fout$. If $f$ contains an inner face $Q \in S$, then $f = Q$. If $f$ contains an inner face $Q \notin S$, then either $Q \not\subset \Omega$ or $Q$ contains a point $q$ with $\dist(q, \partial \Omega) < \gamma^{1/3}$. In the latter case, let $r \in \partial \Omega$ satisfy $|q-r| = \dist(q, \partial \Omega)$. The straight line segment $qr$ cannot intersect any vertex or edge of $G[S]$ because all the points on the segment are too close to $\partial \Omega$. Since $q \in f$, we also have $r \in f$. It follows that $f$ must have nonempty intersection with $\Omega^c = \R^2 \setminus \Omega$, either because $Q$ already intersects $\Omega^c$ or because $r \in f \cap \Omega^c$. The set $\Omega^c$ does not intersect any vertex or edge of $G[S]$, and it is a connected set because $\Omega$ is bounded and simply connected \cite[Section VIII.8]{G01}. Hence a single face of $G[S]$, which must be $\fout$, contains all of $\Omega^c$. We have shown that $f \cap \fout \neq \varnothing$, so $f = \fout$. In conclusion, $G[S]$ has no inner faces besides the elements of $S$. If $H$ is a block of $G[S]$, then Lemma \ref{block faces} implies that all of the inner faces of $H$ are elements of $S$ as well.

Finally, let $e$ be an edge of $G[S]$ that contains a point $p$ with
\[
\dist(p, \partial \Omega) > (9/8) \gamma^{1/3} \, .
\]
We will show that $e$ cannot be an edge of $\fout$, nor can it be an edge of the outer face of a block of $G[S]$. We know that $e$ borders a face $Q \in S$. Since $(9/8) \gamma^{1/3} > \delta$, the point $p$ cannot be part of $\partial G$. Thus $e$ borders two inner faces of $G$, namely $Q$ and another face $Q'$. By \eqref{16-bound},
\[
\dist(Q', \partial \Omega) > \frac{9}{8}\gamma^{1/3} - 2\e \geq \gamma^{1/3}\, .
\]
It follows that $Q' \in S$, so $e$ is not an edge of $\fout$. In addition, both $Q$ and $Q'$ are faces of the same block of $G[S]$, meaning that $e$ cannot be an edge of the outer face of that block (nor any other).

By the argument above, any point $p$ which is on the boundary of $\fout$ or on the boundary of the outer face of a block of $G[S]$ must satisfy
\[
\dist(p, \partial \Omega) \leq (9/8) \gamma^{1/3}\, . \qedhere
\]
\end{proof}

\begin{proof}[Proof of Claim \ref{claim:deriv-Hessian}]
To start, we may replace $g$ with $g-c$ for any constant $c$ without affecting $\nabla h_c$ or $Hh_c$. We choose $c$ so that $g$ vanishes at some point in $\Omega$. Since $\diam(\Omega) = 1$ and the Lipschitz constant of $g$ on the convex set $\widetilde{\Omega}$ is at most $C_1$,
\begin{equation} \label{sup-hc}
\sup_{y \in \Omega} |h_c(y)| \leq \sup_{z \in \partial \Omega} |g(z)| \leq C_1
\end{equation}
where the first inequality is the continuous maximum principle, Proposition \ref{continuous max principle}.

For all $p \in \widehat{G}_1$, we have
\begin{equation} \label{G1-buffer-2}
\dist(p, \partial \Omega) \geq \gamma^{1/3}\, ,
\end{equation}
and for all $q \in \widetilde{G}_1$,
\begin{equation} \label{G1-buffer-3}
\dist(q, \partial \Omega) \geq \gamma^{1/3} - \e \geq \frac{15}{16} \gamma^{1/3}\, ,
\end{equation}
using \eqref{16-bound} in the last inequality. In particular, we have $\widetilde{G}_1 \subset \Omega$.

Proposition \ref{interior derivative estimates} along with \eqref{sup-hc} and \eqref{G1-buffer-2} implies that
\[
\|\nabla h_c\|_{\infty,\widehat{G}_1} \leq CC_1 \gamma^{-1/3}\, .
\]
For the Hessian bound, we observe that since $h_c$ is harmonic, each $Hh_c(y)$ has trace zero and $\|Hh_c(y)\|_2 = |\det Hh_c(y)|^{1/2}$. The individual matrix entries are bounded using Proposition \ref{interior derivative estimates} with \eqref{sup-hc} and \eqref{G1-buffer-3}, and we obtain
\[
\|Hh_c\|_{\infty,\widetilde{G}_1} \leq CC_1 \gamma^{-2/3}\, . \qedhere
\]
\end{proof}

\begin{proof}[Proof of Claim \ref{claim:hd-hd1}]
For any $y \in \Int(V_1^\bullet)$, a function is discrete harmonic at $y$ with respect to $(G^\bullet,c)$ if and only if it is discrete harmonic at $y$ with respect to $(G_1^\bullet,c)$ since the two networks are the same in the immediate neighborhood of $y$. Thus the function $h_d - h_d^{(1)}$ is discrete harmonic on $\Int(V_1^\bullet)$ with respect to $(G_1^\bullet,c)$. By the discrete maximum principle (Proposition \ref{max principle}) and the definition of $h_d^{(1)}$,
\[
|h_d(x) - h_d^{(1)}(x)| \leq \sup_{y \in \partial V_1^\bullet} |h_d(y) - h_d^{(1)}(y)| = \sup_{y \in \partial V_1^\bullet} |h_d(y) - h_c(y)|\, .
\]
Every $y \in \partial V_1^\bullet$ satisfies $y \in \Omega$ and, by Claim \ref{claim:G[S]},
\[
\dist(y, \partial \Omega) \leq (9/8) \gamma^{1/3} < 2\gamma^{1/3}\, .
\]
Claim \ref{claim:boundary} therefore implies that
\[
|h_d(y) - h_c(y)| \leq \frac{C(C_1 + C_2 \e)}{\log^{1/2}(1/\gamma)}
\]
and this completes the proof.
\end{proof}

\begin{proof}[Proof of Claim \ref{claim:A-smooth}]
We know that $x \in V_1^\bullet$. Also, $\dist(x, \partial \Omega) > 2\gamma^{1/3}$ while every point $p \in \partial G_1$ satisfies $\dist(p, \partial \Omega) \leq (9/8)\gamma^{1/3}$ by Claim \ref{claim:G[S]}. Hence the closed disk centered at $x$ of radius $(7/8)\gamma^{1/3}$ is contained in $\Int(G_1)$. Since $(7/8)\gamma^{1/3} \geq 14\e$ by \eqref{16-bound}, we have $B \subset \Int(G_1)$ with room to spare.

Given $y \in A$, we write
\begin{equation} \label{triangle-3}
|f(x) - f(y)| \leq |h_d^{(1)}(x) - h_d^{(1)}(y)| + |h_c(x) - h_c(y)|\, .
\end{equation}
The line segment $xy$ is contained in $\Int(G_1)$, so we may use \eqref{deriv-bound} to see that
\begin{equation} \label{hc-on-A}
|h_c(x) - h_c(y)| \leq \|\nabla h_c\|_{\infty,\widehat{G}_1} (3\e) \leq CC_1 \e^{2/3}\, .
\end{equation}
To bound $|h_d^{(1)}(x) - h_d^{(1)}(y)|$, we use Proposition \ref{equicontinuity}. We have $r = \frac{1}{2}|x-y| \leq 3\e/2$, and we choose $R = (5/8)\gamma^{1/3}$. By \eqref{16-bound}, $R \geq 10\e > 2r + 3\e$. In addition, the closed disk $D$ of radius $R$ centered at $(x+y)/2$ is contained in the closed disk of radius $R + 3\e/2$ centered at $x$. Again using \eqref{16-bound},
\[
R + \frac{3}{2}\e \leq \left( \frac{5}{8} + \frac{3}{32} \right) \gamma^{1/3} < \frac{7}{8} \gamma^{1/3}\, .
\]
Thus $D$ contains no vertices in $\partial V_1^\bullet$, and the value of $\beta$ in Proposition \ref{equicontinuity} is zero. Finally, we compute
\[
\frac{R}{r+\e} \geq \frac{(5/8) \gamma^{1/3}}{(5/2)\e} \geq \frac{1}{4} \e^{-2/3} \geq \e^{-1/3}\, ,
\]
using the $-1/3$ power of \eqref{64-bound} in the last inequality.

Propositions \ref{projection properties} and \ref{sum integral convergence} give
\begin{equation} \label{E1-bound}
\EE_1^\bullet(h_d^{(1)}) \leq \EE_1^\bullet(h_c) \leq 2 \int_{\widehat{G}_1} |\nabla h_c|^2 + 2 \area(\widehat{G}_1)(10LM\e + 8 M^2 \e^2)
\end{equation}
where, by \eqref{deriv-bound} and \eqref{Hessian-bound},
\begin{align*}
L &= \|\nabla h_c\|_{\infty, \widehat{G}_1} \leq CC_1 \gamma^{-1/3}\, , \\
M &= \|Hh_c\|_{\infty, \widetilde{G}_1} \leq CC_1 \gamma^{-2/3}\, .
\end{align*}
We have $\area(\widehat{G}_1) \leq C$. As well, Proposition \ref{energy minimization} implies that
\[
\int_{\widehat{G}_1} |\nabla h_c|^2 \leq \int_\Omega |\nabla h_c|^2 \leq \int_\Omega |\nabla g|^2 \leq CC_1^2 \, .
\]
Plugging these bounds into \eqref{E1-bound} yields
\[
\EE_1^\bullet(h_d^{(1)}) \leq CC_1^2 + C(C_1^2 \e \gamma^{-1} + C_1^2 \e^2 \gamma^{-4/3}) \leq CC_1^2 \, .
\]

We now apply Proposition \ref{equicontinuity} to obtain
\begin{equation} \label{hd-on-A}
|h_d^{(1)}(x) - h_d^{(1)}(y)| \leq \frac{CC_1}{\log^{1/2}(1/\e)} \, .
\end{equation}
Combining \eqref{hc-on-A} with \eqref{hd-on-A} in \eqref{triangle-3},
\[
|f(x) - f(y)| \leq \frac{CC_1}{\log^{1/2}(1/\e)} + CC_1 \e^{2/3} \leq \frac{CC_1}{\log^{1/2}(1/\e)} \, . \qedhere
\]
\end{proof}

\section*{Acknowledgements}
We thank David Jerison for useful discussions. This research is supported by ISF grants 1207/15 and 1707/16 as well as ERC starting grant 676970 RANDGEOM. The second author is supported by a Zuckerman Postdoctoral Fellowship. \\

\footnotesize{
\bibliographystyle{abbrv}
\bibliography{orthodiagonal-final}
 }

\end{document}